\documentclass[11pt]{article}

\usepackage{amsmath,amssymb,amsthm,textcomp}
\usepackage{biblatex} 
\usepackage{bm}
\usepackage{enumerate}
\usepackage{multicol}
\usepackage{graphicx}
\usepackage{wrapfig}
\usepackage{pinlabel}
\graphicspath{ {./Images/ } }
\usepackage[colorlinks=true, allcolors=blue]{hyperref}
\usepackage{fullpage}

\theoremstyle{plain}
\newtheorem{theorem}{Theorem}
\newtheorem{lemma}[theorem]{Lemma}

\newtheorem{corollary}[theorem]{Corollary}

\theoremstyle{plain} 
\newcommand{\thistheoremname}{}
\newtheorem*{genericthm*}{\thistheoremname}
\newenvironment{namedthm*}[1]
  {\renewcommand{\thistheoremname}{#1}%
   \begin{genericthm*}}
  {\end{genericthm*}}

\theoremstyle{definition}
\newtheorem{definition}[theorem]{Definition}

\newcommand{\lra}{\longrightarrow}

\newcommand{\mbf}{\mathbf}

\newcommand{\mbb}{\mathbb}

\makeatletter
\renewcommand{\maketitle}{
\begin{center}
\vspace{2ex}
{\huge \textsc{\@title}}\\
\vspace{3ex}
{\large\textsc{\@author}}
\vspace{1ex}
\end{center}
}
\makeatother

\author{William W. Menasco and Deepisha Solanki}

\begin{document}

\title{Studying links via plats: split and composite links}

\maketitle


\begin{abstract}

Our main results concern changing an arbitrary plat presentation of a split or composite link to one which is obviously recognizable as being split or composite. {\em Pocket moves}, first described in \cite{unlinkviaplats}, are utilized---a pocket move alters a plat presentation without changing its link type, its bridge index or the double coset. A plat presentation of a split link is {\em split} if the planar projection of the plat presentation is not connected. We prove that pocket moves are the only obstruction to representing split links by split plat presentations. Since any pocket move corresponds to a sequence of double coset moves, we have the corollary that the double coset of every plat presentation of a split link has a split plat presentation. We obtain an analogous result for composite links by utilizing {\em flip moves}, which were also first described in \cite{unlinkviaplats}.


\end{abstract}

\section{\centering {Introduction}}
This paper builds on the ideas and singular foliation technology first introduced in \cite{unlinkviaplats}. In that paper, non-oriented links in $\mathbb{R}^3$ were studied using plat presentations. Specifically, the {\em pocket move} and the {\em flip move} on plat presentations of non-oriented links were introduced. Both moves are isotopies that preserve the bridge index of the presentations. For pocket moves it was shown that they correspond to a sequence of moves in the double coset. Flip moves were shown to be equivalent to a precisely depicted sequence of stabilizations, double coset moves, destabilizations. 

We established the following notation.
$\mathcal{L}$ will be a non-oriented link type in $\mathbb{R}^3$ and $L \subset \mathbb{R}^3$ will be an arbitrary representative of $\mathcal{L}$.

Let $L \subset \mathbb{R}^3$ be a representation of $\mathcal{L}$. $\mathcal{L}$ is {\em split} if there exists a $2$-sphere, $S \subset \mathbb{R}^3 \setminus L$, such that $S$ does not bound either a $3$-ball or $\mathbb{R}^3$ minus an open $3$-ball in $\mathbb{R}^3 \setminus L$.  $\mathcal{L}$ is {\em composite} if there exists a $2$-sphere, $S_p \subset \mathbb{R}^3$, twice punctured by $L$ such that neither component of $\mathbb{R}^3 \setminus S_p$ intersects $L$ in a single unknotted arc.

As in \cite{unlinkviaplats}, we will be focusing on representatives $L$ of $\mathcal{L}$ that are $n$-bridge plat presentations.  (The reader should see \cite{birman_1976} for an introduction to plat presentations.) Although our arguments will be mostly geometric it is convenient to utilize the algebra coming from braids to describe the associated $2n$-braid (on $2n$ strands) in terms of the classical Artin generators, $\{ \sigma_i \ | \ 1 \leq i \leq 2n-1 \}$.  For a plat presentation, $L$, we will use the notation, $W(L)$, for the associated braid word, written in terms of the $\sigma_i \text{'s}$.

A plat presentation, $L$, is {\em split} if $W(L)$ is missing a $\sigma_{2i}$ for $1 \leq i < n$.  For such a braid word one can readily see the associated splitting $2$-sphere in a regular planar diagram of the plat, $L$---such a sphere will intersect the plane containing the diagram in a circle that splits the diagram.

A plat presentation, $L$, is {\em composite} if $W(L)$ is missing a $\sigma_{2i+1}$, $1 \leq i < n-1$.  For such a braid word one can readily see a composing sphere that intersects the regular planar diagram of $L$ in a circle that intersects the diagram twice, once at a ``top'' and once at a ``bottom'' bridge.



The generic situation is that the $2$-sphere illustrating that a link type is split or composite is ``obscured'' for an arbitrary plat presentation.  It is through the uses of pocket and flip moves that we can take such an obscured $2$-sphere and make it readily observable as it would be if the plat presentation was split or composite.  Specifically, we have the following main theorems:




\begin{theorem}
Let $\mathcal{L}$ be a split link type. Let $L$ be an $n$-bridge plat presentation of $\mathcal{L}$. Then, there exists a finite sequence of plat presentations of $\mathcal{L}$:
$$ L = L_0 \lra L_1 \lra L_2 \lra \cdots \lra L_k $$
such that $L_k$ is split and $L_{i+1}$ is obtained from $L_{i}$ via the following \textbf{constant bridge index} moves: 
\begin{itemize}
\item[(i)] a braid isotopy on the 2n strands;
\item[(ii)] the pocket move.
\item[(iii)] double coset moves.
\end{itemize}
\label{thm1}
\end{theorem}


\begin{theorem}
Let $\mathcal{L}$ be a composite link. Let $L$ be an $n$-bridge plat presentation of $\mathcal{L}$. Then, there exists a finite sequence of plat presentations of $\mathcal{L}$:
$$ L = L_0 \lra L_1 \lra L_2 \lra \cdots \lra L_k $$
such that $L_k$ is a composite plat and $L_{i+1}$ is obtained from $L_{i}$ via the following \textbf{constant bridge index} moves: 
\begin{itemize}
\item[(i)] a braid isotopy on the 2n strands;
\item[(ii)] the pocket move;
\item[(iii)] the flip move;
\item[(iv)] double coset moves 
\end{itemize}
\label{thm2}
\end{theorem}

For an $n$-bridge plat presentation $L$ of the non-oriented link type $\mathcal{L}$ we use $\mathbf{DC}(L)$ to denote the collection of all $n$-bridge plat presentations in the same double coset as $L$.  We then have the following corollary to Theorem \ref{thm1}:

\begin{corollary}
    \label{corollary: double cosets}
    For every plat presentation $L$ of a split link, there exists a plat presentation $L^\prime \subset \mathbf{DC}(L)$ such that $L^\prime$ is a split plat presentation.
\end{corollary}

The corollary follows from Theorem \ref{thm1} and Lemma 3 of \cite{unlinkviaplats} that pocket moves can be realized as a sequence of double coset moves.

\subsection{Idea of proof}

Our strategy for establishing Theorems \ref{thm1} and \ref{thm2} is inspired by the singular foliation technology coming from the Birman-Menasco papers. There they considered an appropriate essential surface in the complement of the closed braid and studied the singular foliation induced by the braid fibration.

In our setting we start with pairs ${(S,L)}$-splitting sphere and plat presentations of the split link and ${(S_p,K)}$-twice punctured sphere and plat presentations of the composite knot. We then consider the singular foliation induced on $S$, $S_p$ coming from the plat height function. The  singular foliation technology used for proving Theorems \ref{thm1} and \ref{thm2} will have the following salient components: 

(i) A system of level curves corresponding to the regular values of the height function (See \S \ref{level curves}) \\
(ii) Generic singular leaves corresponding to the critical values of the height function. The foliated neighborhood of each singular leaf will form a {\em singular tile}.  These tiles will give a decomposition of $S$ and $S_p$. (See \S \ref{tilingandfoliation}) \\
(iii) The tiling decomposition of $S$ or $S_p$ will imply an Euler characteristic equation (See Lemma \ref{surfaceec}), this equation is used to obtain a combinatorial relation between different types of tiles. (See Lemmas \ref{bla} and \ref{bla_puncture})\\
(iv) The tiling of $S$ or $S_p$ will be ``dual'' to a finite graph that allows us to define a complexity measure on the tiling. (See \S \ref{graphcomplexity}) \\
(v) The local geometric realization of certain tiling patterns will imply the plat admits a pocket move or a flip move. Applying a pocket move, flip move or destabilisation will strictly reduce the complexity measure. 

In \S \ref{tilingandfoliation}, we describe the tiles associated with the foliation, each tile type will contain a single singularity, or in the case of $S_p$, a single singularity or puncture. 

This gives an obvious complexity measure on the pairs ${(S,L)}$, ${(S_p,K)}$, counting the number of singularities of certain types. This is rigorously defined in \S \ref{graphcomplexity}. 

In \S \ref{tilingandfoliation}, we explain how the ``configuration" in the singular foliation corresponds to the plat moves and results in \S \ref{mainresult} detail how plat moves remove these singularities, thus reducing the complexity measure on $S$ or $S_p
$.

\subsection{Contrasting Plats with Closed Braids}

Plats and closed braids are both objects relating braids to links, and, at first glance, appear very similar. But there are indeed considerable differences between these two, which we elucidate below: 

\begin{itemize}

\item \textbf{Non-oriented vs Oriented Links:} The definition of closed braids refers to the winding around of the link around the braid axis, thus inherently lending the link an orientation. Therefore, closed braids are a tool to study {\em oriented link types}. On the other hand, plats have no natural axis, and thus no restriction to being oriented. Thus, plats are a tool to study {\em non-oriented link types}.

\item \textbf{Bridge number vs Braid index:} Plats and closed braids offer us two different complexity measures on links: {\em bridge number} and {\em braid index}. The {\em bridge number} of a link type $\mathcal{L}$ is the smallest integer $n$ such that $\mathcal{L}$ can be represented as a $2n$- plat and, the {\em braid index} of a link type $\mathcal{L}$ is the least number of strings required to express $\mathcal{L}$ as a closed $n$- braid. These two measures of complexity inherently capture different types of information about a link. This can be seen from the fact that the braid index of a link type $\mathcal{L}$ is equal to the minimum number of Seifert circles of $\mathcal{L}$ and the bridge number of $\mathcal{L}$ can alternatively defined to be the minimum numbers of points of minima (or maxima) in any diagram of $\mathcal{L}$. 

\item \textbf{Double Coset Classes vs Conjugacy Classes:} Every link can be expressed as both a closed braid and a plat in infinitely many different ways, which leads to an immediate question: {\em How are any two braids representing the same link type as closed braids (respectively plats) algebraically related to each other?} Markov's theorem (\cite{braidslinksandmcgs}, \cite{bm_markov}) answers this question in the case of closed braids and Birman's result about Stable Equivalence of Plats (\cite{birman_1976}) answers this in the case of plats. Consider the partition of $\mathcal{B}_n$ into conjugacy classes: $\mathcal{B}_n = \bigcup\limits_{\alpha} \mathcal{C}_{\alpha}$. Then, each conjugacy class $\mathcal{C}_{\alpha}$ corresponds to the same knot type, but this correspondence is not one-to-one. To delve into the case of plats, we need to introduce the Hilden subgroup $\mathcal{K}_{2n}$ (\cite{hilden_sbgp}) of $\mathcal{B}_{2n}$. $\mathcal{K}_{2n}$ consists of all those braids which can be concatenated before or after a braid without changing its link type as a plat i.e. the plat given by the braids $b_1 b b_2$ and $b$ correspond to the same knot type where $b_1, b_2 \in \mathcal{K}_{2n}$. Thus, $\mathcal{B}_{2n}$ is partitioned into double cosets modulo $\mathcal{K}_{2n}$ on both sides ($\mathcal{B}_{2n} = \bigcup_{b \in \mathcal{B}_{2n}} \mathcal{K}_{2n} b \mathcal{K}_{2n}$) such that each double coset corresponds to the same knot type, and again, this correspondence is not one-to-one.

\end{itemize}

\subsection{Outline of the paper}

In \S \ref{prelims}, we introduce the basic terminology and set up the required machinery. \S \ref{prelims} begins with a description of the system of level curves in \S \ref{level curves}. This is followed by an explanation of moves on plats in \S \ref{platmoves}. Then, we have a detailed description of the different types of singularities that occur on $S$ and $S_p$ (with respect to the plat height function), and how that set up is used to foliate the sphere and construct a tiling of it, in \S \ref{tilingandfoliation}. \S \ref{prelims} concludes with an introduction to the graph corresponding to a sphere foliation and how that is used to define a complexity function on the sphere, in \S \ref{graphcomplexity}. 

Moving on to the next section, \S \ref{mainresult} consists of seven lemmas leading up to the proof of Theorems \ref{thm1} and \ref{thm2}. Lemma \ref{surfaceec} gives the relation between the Euler characteristics of a surface $X$ and subsurfaces $X_i$'s such that the $X_i$'s are glued together to give $X$. We use this formula to obtain a combinatorial relation between different types of tiles in Lemma \ref{bla}, pertaining to $S$ and Lemma \ref{bla_puncture}, which pertains to $S_p$. Lemma \ref{bla} is crucial in helping us understand the constraints on how different singularities on $S$ and $S_p$ can grow with respect to each other. It also helps us identify some key characteristics of the graph associated with the foliation of the spheres. This leads us to Lemma \ref{lemma5} which tells us how to start eliminating singularities from the spheres, using the language of graph theory. In particular, in Lemma \ref{lemma5}, we give conditions that a vertex in the graph needs to satisfy so that the singularity it represents can be removed using one of the moves in Theorems \ref{thm1} or \ref{thm2}, depending on the context, giving us a plat with a strictly smaller value of the complexity function. In Lemma \ref{lemma6}, we prove that a vertex satisfying conditions of Lemma \ref{lemma5} always exists, if there are any singularities of a certain type on the sphere. 

The rest of \S \ref{mainresult} is dedicated to proving the statement of Theorem \ref{thm1}. \S \ref{mainresult} concludes with Lemma \ref{remove_punctured_tiles}, which pertains only to $S_p$. It describes how to deal with the special tile $T_p$ (containing a puncture) in different scenarios.

\section*{Acknowledgements} The authors would like to thank Seth Hovland, Greg Vinal and Hong Chang for their many discussions on plat presentations.


\begin{figure}[ht!]
\labellist
\small\hair 2pt
\pinlabel {$n$-bridges} at 208 460
\pinlabel {$n$-bridges} at 205 22
\pinlabel {$\textbf{B}$} at 204 235
\pinlabel {$2n$-strands} at 450 298
\endlabellist
\centering
\includegraphics[width=9cm, height=7cm]{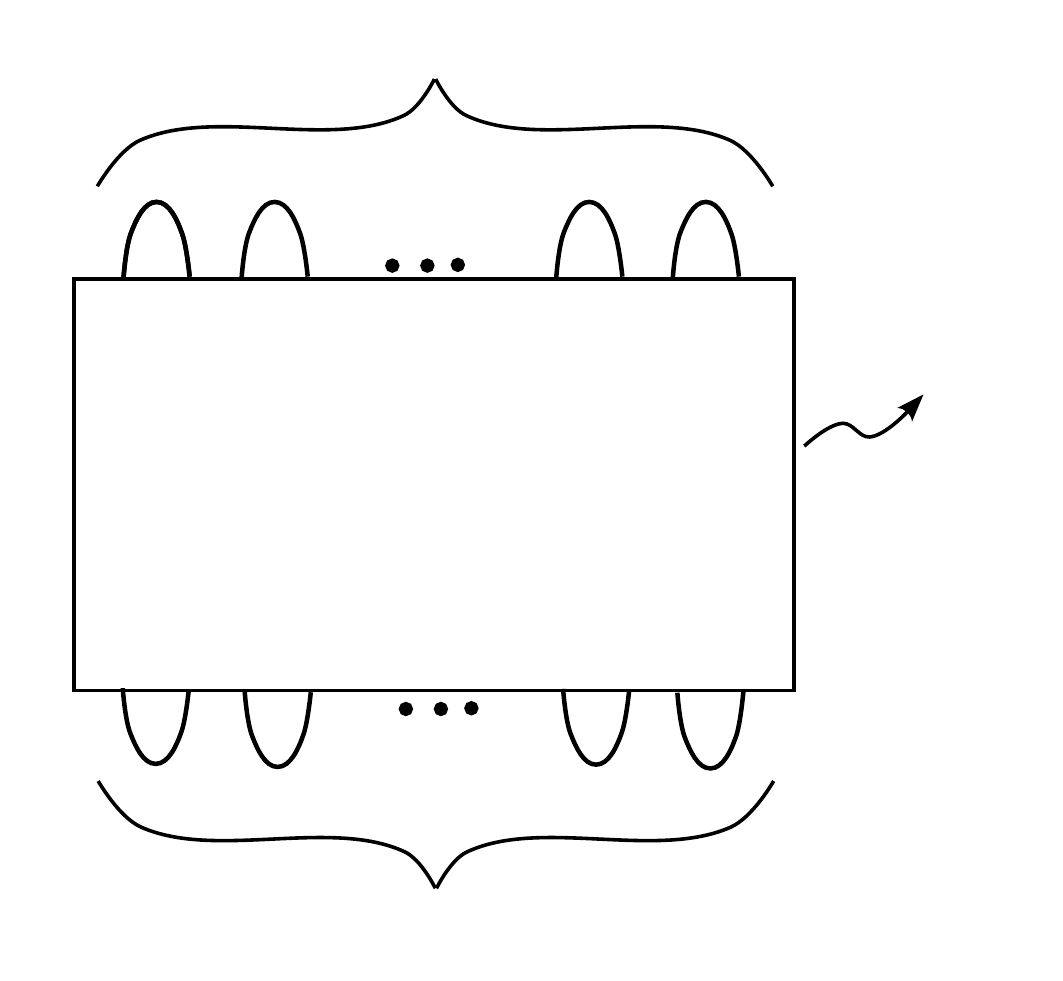}
\caption{A generic 2n-plat}
\label{plat}
\end{figure}
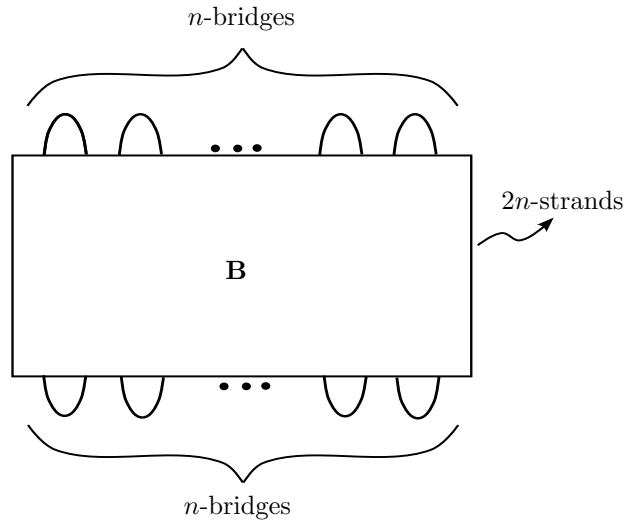

\section{\centering Preliminaries}

\label{prelims}





Starting with plat presentations for non-oriented links, we give a brief description of what they are. Consider the classical braid group on $2n$ strands, $\mathcal{B}_{2n}$. Then, for an element of $\mathcal{B}_{2n}$ i.e. a braid $B$ on $2n$ strands, we cap-off the $i$ and $i+1$ strands, for odd $i$ between $1$ and $2n-1$, at the top (bottom) of the braid with $n$ northern (southern) circle hemispheres, refer Figure \ref{plat}. This gives us a link, and this presentation of the link is called a {\em plat presentation}. Call this the {\em plat defined by $B$}. Then $n$ is the {\em bridge index} of the presentation. We thus have a height function on the presentation that has $n$ maxima (minima) points all at the same level 
above (below) the braid. (We will employ this height function extensively in the arguments of this note). One can jump/isotope from one plat presentation to another through braid isotopies, which again are Reidemeister type-II \& -III moves placed in this setting. Braid isotopies do not change the bridge index. An isotopy by a stabilization (destabilization) adds (removes) two strands in the braid along with a maxima and minima, which is the Reidemeister type-I move placed in the plat setting. These isotopies are shown in Figures \ref{braidisotopy_one}, \ref{braidisotopy_two} and \ref{stab}.

Birman showed in \cite{birman_1976} that any link type can be represented as a plat. It turns out that this representation is highly non unique and every link type has {\em infinitely many} distinct plat presentations. By ``distinct plat presentations'', we mean that the braids corresponding to these plat presentations correspond to different elements in the braid group. In the same paper, Birman proved the following result, which tells us how any two braids, whose plats represent the same knot type, are algebraically related to one another:

\begin{namedthm*}{Birman's Result}

Let $\mbf{L_i}$, i =1, 2, be tame knots. Choose elements $\phi_i \in \mathcal{B}_{2n_i}$ in such a way that the plat defined by $\phi_i$ corresponds to the same knot type as $\mbf{L_i}$. Then, $\mbf{L_1} \approx \mbf{L_2} $ if and only if there exists an integer $t \geq$ max$(n_1, n_2)$ such that, for each $n \geq t$, the elements 
$$\phi_{i}^{'} = \phi_i \sigma_{2n_i} \sigma_{2n_i +2} ... \sigma_{2n-2} \in \mathcal{B}_{2n}, i = 1, 2$$
are in the same double coset of $\mathcal{B}_{2n}$ modulo the subgroup $\mathcal{K}_{2n}$.

\end{namedthm*}

Birman's Result can be restated as follows: 

\begin{namedthm*}{\textbf{Stable equivalence of plat presentations}}

Let $K$ be a $2n$-plat presentation of $\mbf{K}$, and $K'$ be a $2n^{\prime}$-plat presentation of $\mbf{K}$. Then, there is a finite sequence of plat presentations of $\mbf{K}$: 
$$ K = K_{0} \longrightarrow K_{1} \longrightarrow K_{2} \longrightarrow ... \longrightarrow K_{m} = K' $$
such that $K_{i+1}$ is obtained from $K_{i}$ via the following moves:
\begin{itemize}
\item[(i)] braid isotopies;
\item[(ii)] double coset moves;
\item[(iii)] stabilization or destabilization.
\end{itemize}
\end{namedthm*}

\begin{figure}[ht!]
\centering
\labellist
\small\hair 2pt
\pinlabel {R II} at 59 108
\pinlabel {R II} at 307 108
\endlabellist
\centering
\includegraphics[width=10cm, height=2.5cm]{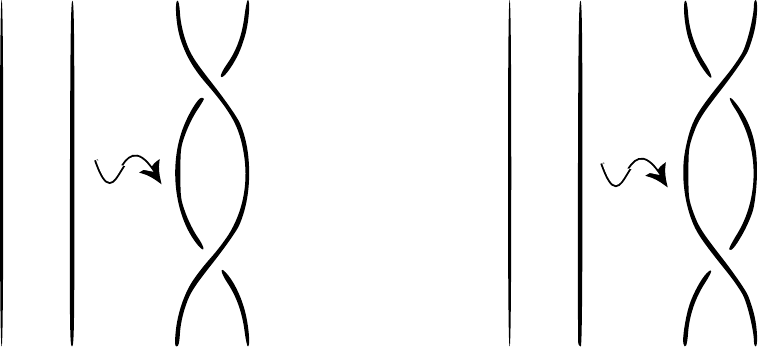}
\caption{Braid isotopy corresponding to Reidemeister II move}
\label{braidisotopy_one}
\end{figure}

\begin{figure}[ht!]
\centering
\labellist
\small\hair 2pt
\pinlabel {R III} at 125 110
\endlabellist
\centering
\includegraphics[width=6cm, height=2.5cm]{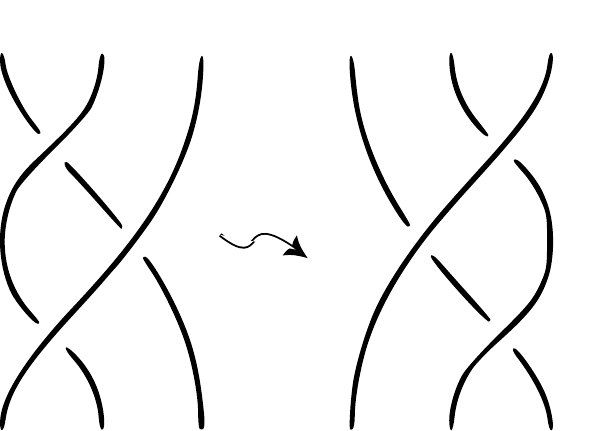}
\caption{Braid isotopy corresponding to Reidemeister III move}
\label{braidisotopy_two}
\end{figure}

\begin{figure}[ht!]
\centering
\labellist
\small\hair 2pt
\pinlabel {$\mbf{B}$} at 120 127
\pinlabel {$\mbf{B}$} at 428 127 
\pinlabel {$\sigma_{2n}$} at 529 26
\pinlabel {STABILIZATION} at 282 146
\pinlabel {DESTABILIZATION} at 282 100
\endlabellist
\centering
\includegraphics[width=18.5cm, height=7cm]{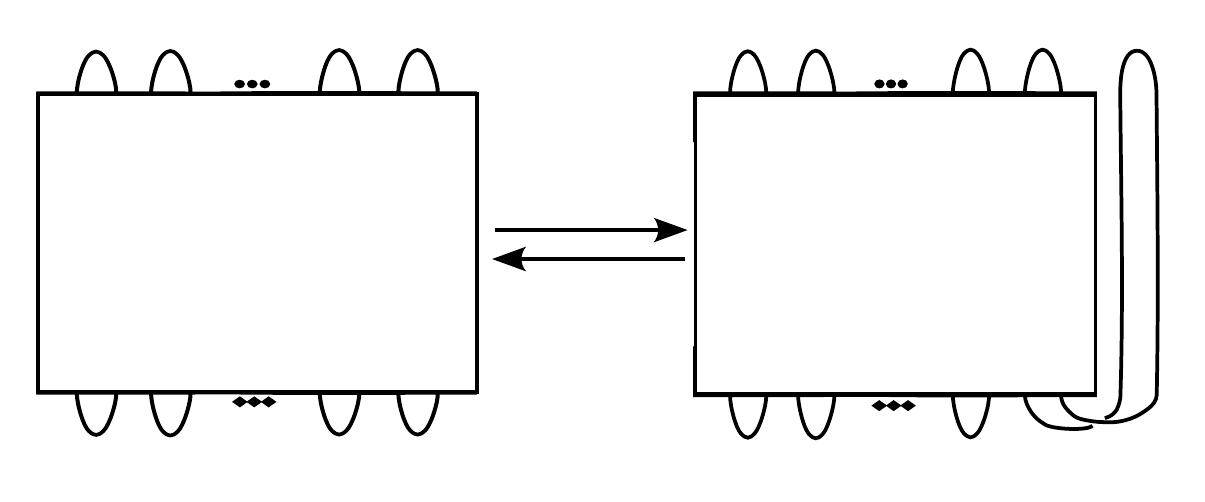}
\caption{Stabilizing and Destabilizing a Plat}
\label{stab}
\end{figure}


Let $L_i$ be a plat presentation of $\mathcal{L}$, from the sequences in Theorems \ref{thm1} and \ref{thm2}, and let $S_i$ be a geometric realisation of $S$ or $S_p$ corresponding to $L_i$, depending on whether $L_i$ is split or composite. We assign a height function $h$ to $S_i$ such that $h$ is a Morse function, regarded as the projection function from $S_i$ onto $\mbb{R}$ (think of this as the natural height function $h: \mbb{R}^3 \lra \mbb{R}$. We position $S_i$ in $\mbb{R}^3$ so that $h$ takes values on $S_i$ varying between $-0.25$ and $1.25$ such that, the points of maxima of the top bridges are at height $h = 1$ and the points of minima of the bottom bridges are at height $h=0$. With this setup, thinking of $\mbb{R}^3$ as $\mbb{R}^2 \times \mbb{R}$, then $\mbb{R}^2 \times \{*\}$ are the level planes for $h$, say $P_t = h^{-1}(t) =  \mbb{R}^2 \times \{t\}$, for $t \in [-0.25, 1.25]$.

\subsection{System of Level Curves} \label{level curves}

Define $C_t^i := S_i \cap P_t$. Then, $C_t^i$ is the set of level curves of $S_i$ at level $t$, which is basically a cross section of $S$ or $S_p$ at height $t$. Note that, if $L_i$ is a plat of index $n$, then, for $ t \in (0,1)$, $C_t^i$ consists of $m_t^i$ closed curves, say. Call these closed curves $s_{i,t}^j$, for $j \in \{1,2, ..., m_t^i\}$. We position $S_i$ so that any points of local maxima or minima occur for $t$-values {\em not} in $(0,1)$. Also notice that, for $t \in (0,1)$, the link/knot intersects $P_t$ in exactly $2n$ points. If, for any $i,j,t$, $s_{i,t}^j$ does not contain a part of the link/knot or another curve inside of it (i.e. $s_{i,t}^j$ is empty, which means that it does not contain points or other curves), we surger it off. This means that, for any $i$ and any $t \in [-0.25, 1.25]$, the only closed curves that remain on $S_i$ are the ones that contain a part of the link/knot or another simple closed curve. Further, notice that, for all but finitely many values of $t$, the closed curves $s_{i,t}^j$ are all simple closed curves. The values of $t$ for which we get intersection (which includes self-intersection) between closed curves are the values where the sphere $S_i$ has a saddle. We position the sphere so that no value of $t$ has more than one saddle.

\subsection{Moves on plats} \label{platmoves}

\subsubsection{Braid isotopies} \label{braid_isotopies}

Braid isotopies are moves on plats where the top and bottom bridges (i.e. the local max's and min's) are fixed. Algebraically, this is equivalent to changing the braid word corresponding to the plat using the relators in the braid group. So the end points of the strings are all fixed and with that constraint, any manipulation of strings is allowed as long as the resulting object is still a plat. This is shown in Figure~\ref{braidisotopy_one} and Figure~\ref{braidisotopy_two}.




\subsubsection{Double coset moves and the Pocket move}

\label{dcms}

Double coset moves on plats permute the bridges in different ways, thus giving freedom of movement to the bridges. For the braid group $\mathcal{B}_{2n}$, a double coset move on the plat defined by $\sigma \in \mathcal{B}_{2n}$ is one which takes this plat to the plat defined by $\alpha \sigma \beta$  where $\alpha$ and $\beta$ are elements in $\mathcal{K}_{2n}$, the Hilden subgroup (\cite{hilden_sbgp}) of the braid group, which has the following generating set: 
$\{ \sigma_{1}, \sigma_{2} {\sigma_{1}}^{2} \sigma_{2}, \sigma_{2i} \sigma_{2i-1} \sigma_{2i+1} \sigma_{2i}, 1 \leq i \leq n-1 \}$. The relators for $\mathcal{K}_{2n}$ were computed by Stephen Tawn, in \cite{tawn2007presentation}. 

For $n=2$, the generators of the Hilden subgroup are: $\{ \sigma_{1}, \sigma_{2} {\sigma_{1}}^{2} \sigma_{2}, \sigma_{2} \sigma_{1} \sigma_{3} \sigma_{2} \}$. The plat moves on the bottom bridges coming from these generators are shown in Figure~\ref{hildengens}. 

Note that the double coset moves essentially try to capture all those moves, not covered by the scope of braid isotopies and stabilization/destabilization, which take plats to plats while preserving the link type. Figure~\ref{hildengens} is a good motivating example to observe this phenomenon. Notice that all the moves in Figure~\ref{hildengens}, change the braid word but not the knot type.

Birman, in \cite{birman_1976}, does not use the nomenclature, ``double coset moves". We are using it here, because, without this class of moves, given two plat presentations $K_1$ and $K_2$ of a knot, using \textbf{only} braid isotopies and stabilization/destabilization, the best we can do is get $K_1$ and $K_2$ to be in the same double coset of $\mathcal{B}_{2n}$ modulo $\mathcal{K}_{2n}$. 

\begin{figure}[ht!]
\centering
\labellist
\small\hair 2pt
\endlabellist
\centering
\includegraphics[width=12.5cm, height=3.5cm]{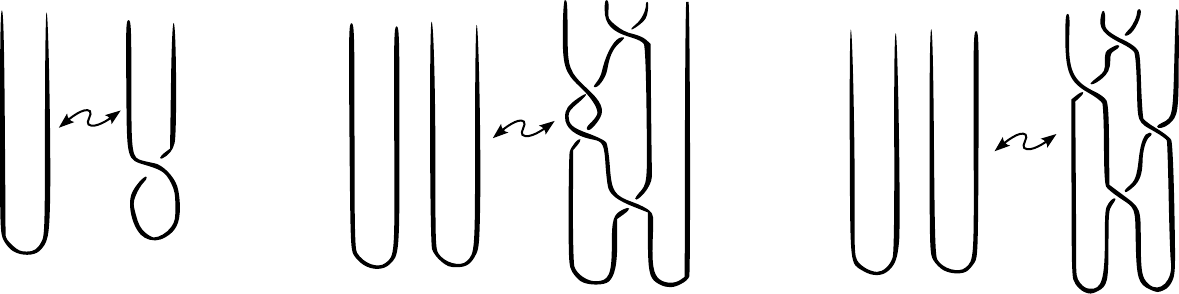}
\caption{Generators of $\mathcal{K}_4$}
\label{hildengens}
\end{figure}

\begin{figure}[ht!]
\centering
\labellist
\small\hair 2pt
\pinlabel {A} at 224 402
\pinlabel {B} at 224 118
\pinlabel {A} at 585 402
\pinlabel {B} at 585 118
\endlabellist
\centering
\includegraphics[width=15cm, height=8.3cm]{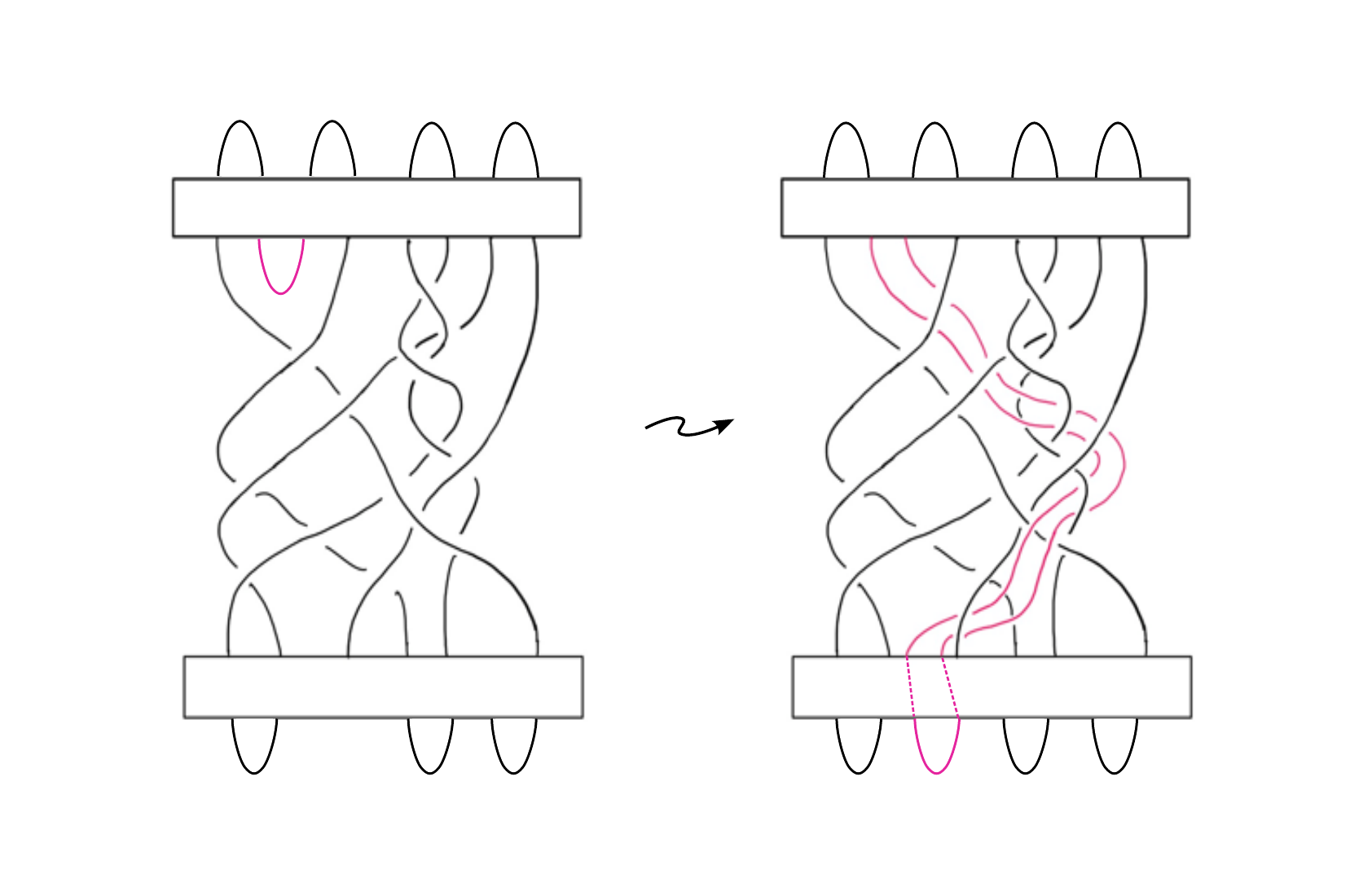}
\caption{Pocket move on an 8-plat}
\label{pocket_general}
\end{figure}

\begin{figure}[ht!]
\centering
\labellist
\small\hair 2pt
\pinlabel {A} at 153 361
\pinlabel {B} at 143 85
\pinlabel {A} at 444 361
\pinlabel {B} at 458 85
\pinlabel {C} at 118 313
\pinlabel {C} at 423 39
\endlabellist
\centering
\includegraphics[width=15cm, height=8.5cm]{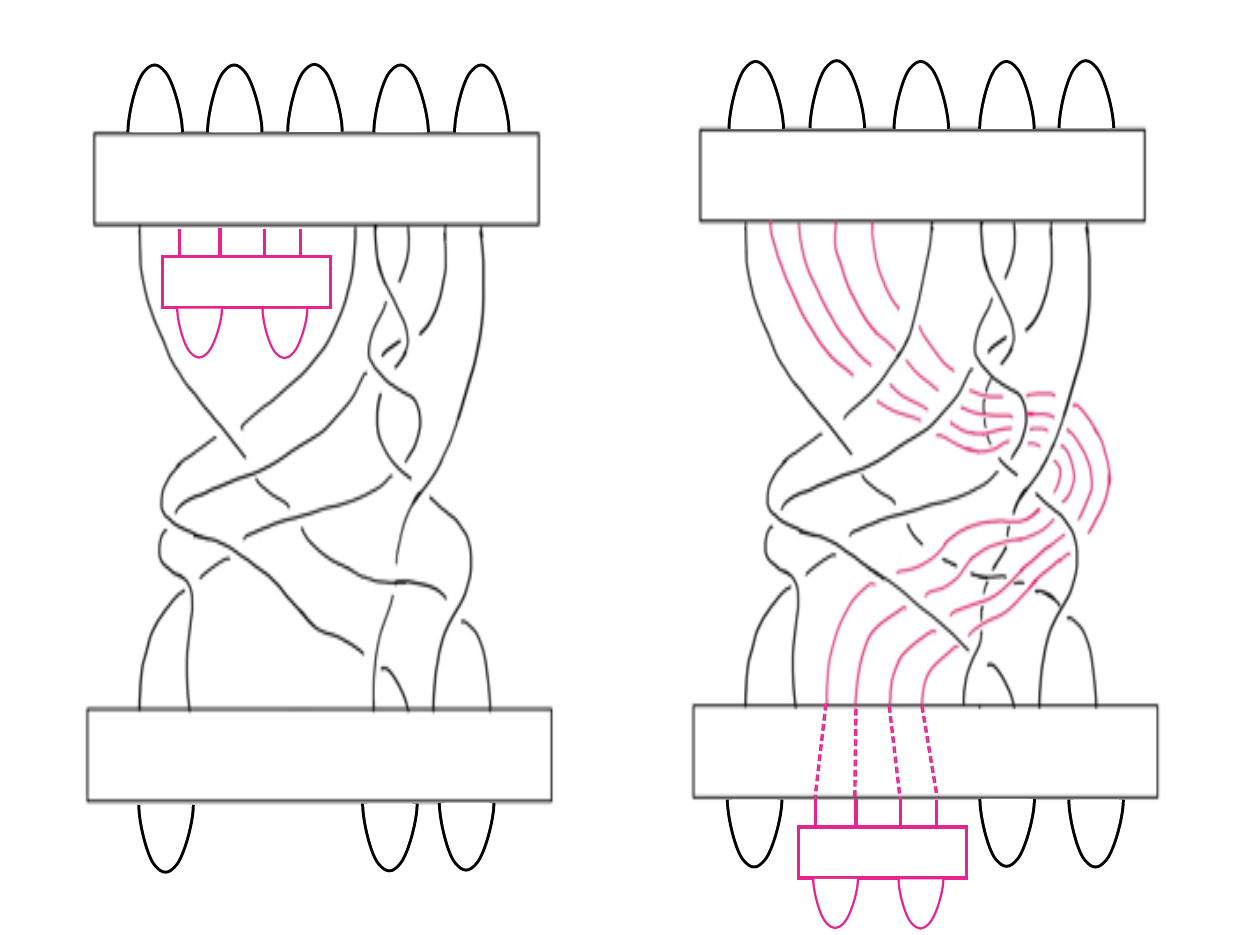}
\caption{Generalised pocket move on a higher index plat}
\label{block_pocket}
\end{figure}

Figures \ref{pocket_general} and \ref{block_pocket} depict a generic pocket move. Consider a plat presentation $L$ where a top or a bottom bridge which is truncated at a level strictly between the top and the bottom levels of the plat (as shown in the Figures), say $t = t_0$, where $0 < t_0 < 1$. Now, define a smooth path $\gamma$, $\gamma: [0,1] \lra \mbb{R}^3 \ L$, such that the plat height function is monotonic on $\gamma$ and we have that $h^{-1}(\gamma(0)) = t_0$ and $h^{-1}(\gamma(1)) = 1$, in case we are doing a pocket move with a bottom bridge, or $h^{-1}(\gamma(1)) = 0$, if we are doing a pocket move with a top bridge. Now, we take the local max or min corresponding to the bridge that we picked and drag it along the path $\gamma$. This results in the bridge weaving around the rest of the plat, as shown in Figure \ref{pocket_general}. Further, we can also twist the bridge as it follows the path $\gamma$, and we can do this process for any top or bottom bridge that can be truncated prematurely, with each bridge following a potentially different path. In \cite{unlinkviaplats}, we prove the following Lemma:

\begin{lemma}
   
Pocket moves can be realised as a combination of braid isotopies and a sequence of double coset moves.

\label{pmisdcm}
\end{lemma}

This lemma shows us the geometric scope of the pocket move.

\subsubsection{The flip move}

\label{flipmove}

Consider the unknot positioned in 3-space given by the following equation: $K = \{(x,y,z) : x^2 + \frac{y^2}{4} = 1, z = 0\}$. Figure ~\ref{flip} shows the \textbf{flip move on the index one unknot diagram}, which is the following isotopy (or its inverse): take point B ($(0,-2,0)$) and move it along the following path: 
$$ x = 0, y = \cos(\theta) -1, z = \sin(\theta)$$ where $ \pi \leq \theta \leq 3\pi$. 


This path traces out the circle of radius one centered at $(0,-1,0)$, in the $y-z$ plane, moving clockwise if viewed from the region $x \geq 0$. The effect of this isotopy on the disc is shown in Figure~\ref{flip}.

\begin{figure}[ht!]
\centering
\labellist
\small\hair 2pt
\pinlabel $A$ at 132 368 
\pinlabel $A$ at 221 385
\pinlabel $A$ at 388 376
\pinlabel $A$ at 531 391
\pinlabel $B$ at 136 124
\pinlabel $B$ at 225 149
\pinlabel $B$ at 402 184
\pinlabel $B$ at 519 212
\pinlabel {saddle} at 538 279
\pinlabel {pt of minima} at 386 127
\pinlabel {pt of minima} at 524 162
\endlabellist
\centering
\includegraphics[scale=0.65]{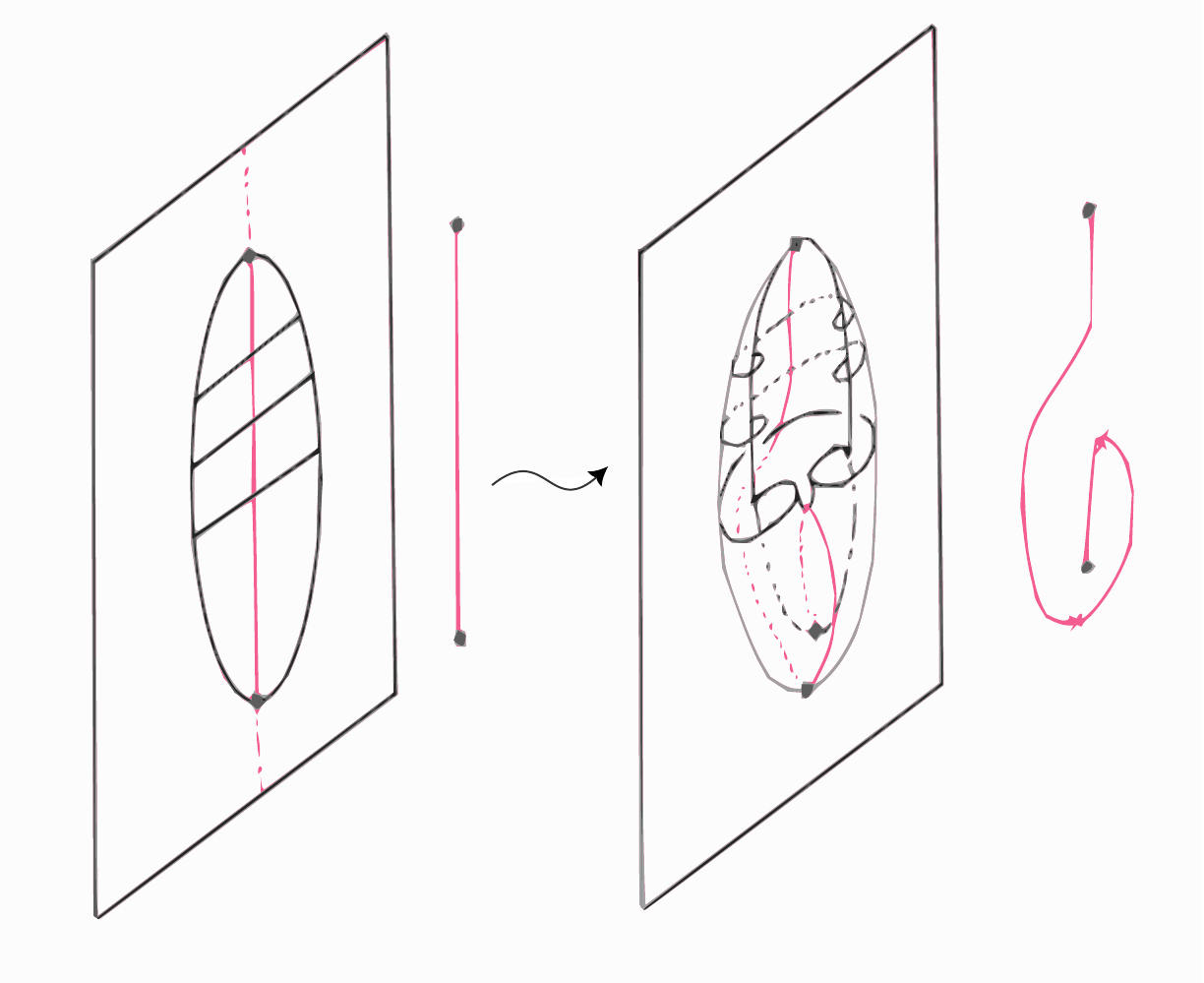}
\caption{The flip move on the 1-bridge 0-crossing diagram}
\label{flip}
\end{figure}
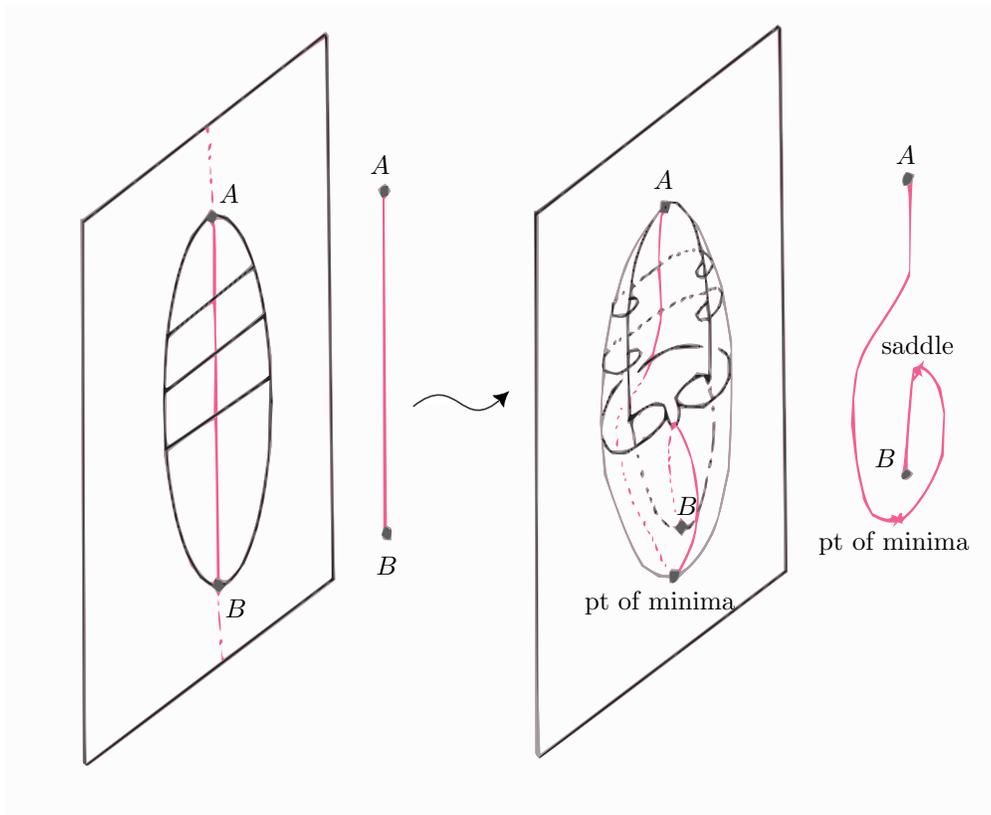

\begin{figure}[ht!]
\centering
\labellist
\small\hair 2pt
\pinlabel $A$ at 154 311 
\pinlabel $A$ at 472 311
\pinlabel $B$ at 157 102
\pinlabel $B$ at 474 102
\endlabellist
\centering
\includegraphics[width=12cm, height=6.5cm]{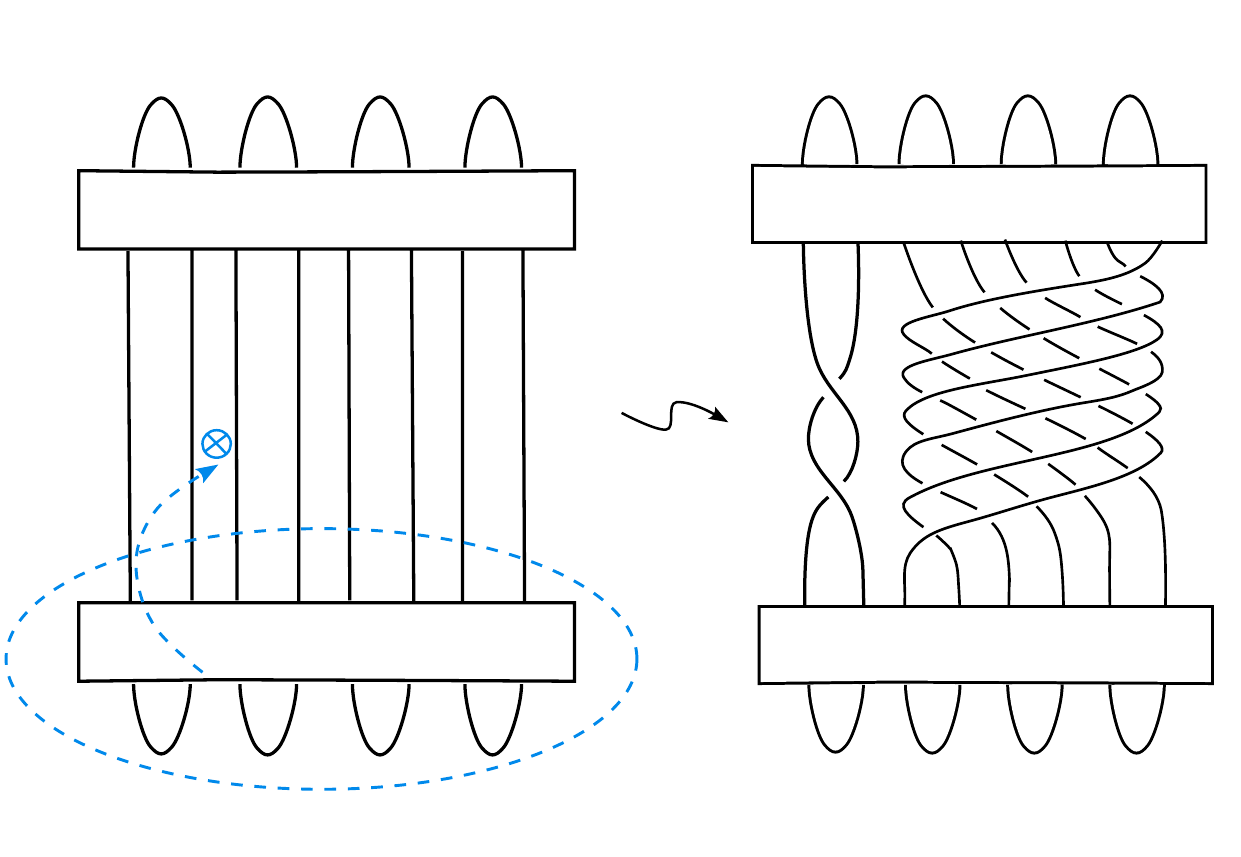}
\caption{The flip move on an 8 plat}
\label{flip_8}
\end{figure}

On a generic $2n$-plat with braid word $\mathbf{AB}$, here's how the flip move changes the braid word: \\
Case (i): Flipping $\mathbf{B}$ into the plane of the paper between the $k$-th and $(k+1)$-st strands: 
$$ \mathbf{AB} \lra \mathbf{A} {(\sigma_1 \sigma_2...\sigma_{k-1})}^{k} {((\sigma_{2n-1})^{-1}...(\sigma_{k+2})^{-1}(\sigma_{k+1})^{-1})}^{2n-k} \mathbf{B} $$ 
Case (ii): Flipping $\mathbf{B}$ out of the plane of the paper between the $k$-th and $(k+1)$-st strands: 
$$ \mathbf{AB} \lra \mathbf{A} {((\sigma_{k-1})^{-1}(\sigma_{k-2})^{-1}...(\sigma_{1})^{-1})}^{k} {(\sigma_{k+1} \sigma_{k+2}... \sigma_{2n-1})}^{2n-k} \mathbf{B} $$ 
Case (iii): Flipping $\mathbf{B}$ into the plane of the paper between the first two strands: 
$$ \mathbf{AB} \lra \mathbf{A} {((\sigma_{2n-1})^{-1} (\sigma_{2n-2})^{-1}... (\sigma_{3})^{-1} (\sigma_{2})^{-1})}^{2n-1} \mbf{B} $$ 
Case (iv): Flipping $\mbf{B}$ out of the plane of the paper between the first two strands: 
 $$ \mbf{AB} \lra \mbf{A} {(\sigma_2 \sigma_3... \sigma_{2n-2} \sigma_{2n-1})}^{2n-1} \mbf{B} $$
Case (v): Flipping $\mbf{B}$ into the plane of the paper between the last two strands: $$ \mbf{AB} \lra \mbf{A} {(\sigma_1 \sigma_2 ... \sigma_{2n-3} \sigma_{2n-2})}^{2n-1} \mbf{B} $$
Case (vi): Flipping $\mbf{B}$ out of the plane of the paper between the last two strands: 
$$ \mbf{AB} \lra \mbf{A} {((\sigma_{2n-2})^{-1} (\sigma_{2n-3})^{-1} ... (\sigma_{2})^{-1} (\sigma_{1})^{-1})}^{2n-1} \mbf{B} $$ 

Further, we define the \textbf{micro flip move}, wherein only $k$ (where $k$ is even) out of the $n$ strands are flipped. A pictorial representation is given in Figure~\ref{microflip}.

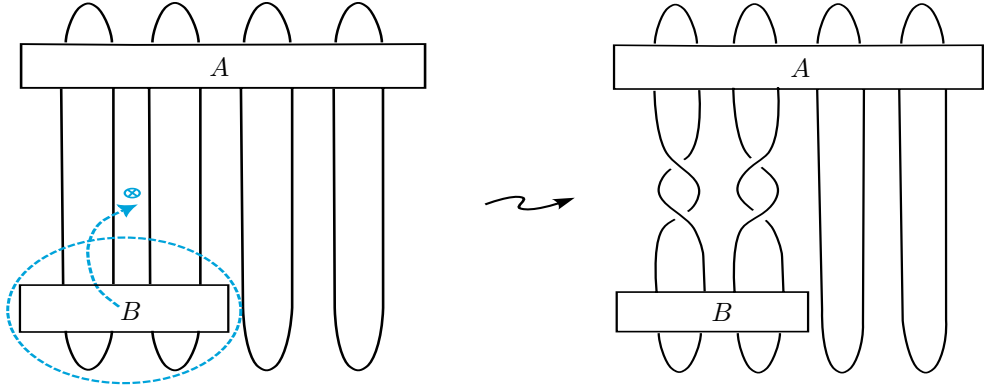
\begin{figure}[ht!]
\centering
\labellist
\small\hair 2pt
\pinlabel $A$ at 134 283
\pinlabel $A$ at 475 282
\pinlabel $B$ at 82 85
\pinlabel $B$ at 429 84
\endlabellist
\centering
\includegraphics[width=13.5cm, height=6cm]{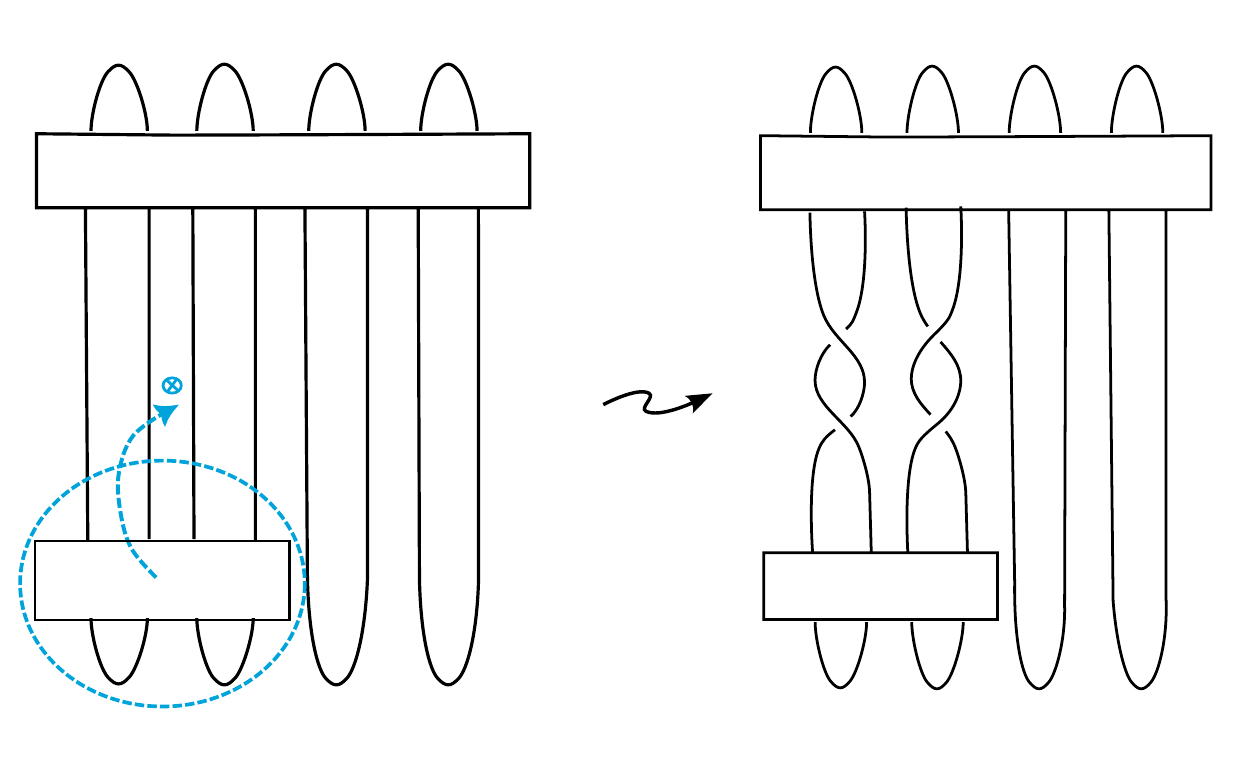}
\caption{An example of a Microflip move}
\label{microflip}
\end{figure}

Corresponding to the aforementioned cases, we also have six more cases where we do the flip with the top bridges. The change in braid word can be computed as shown above.
Figure~\ref{stab_seq} shows the corresponding sequence of moves using only braid isotopies, double coset moves and stabilization/destabilization, to go from the plat on the right to the plat on the left in Figure~\ref{flip_8}. Repeating the same sequence of moves on the first strand (from the left) in the last plat in Figure~\ref{stab_seq}, we get the left plat in the first row in Figure~\ref{garside_slide}. This plat is now defined by the braid $A \delta B$, where $\delta$ is the full twist in the braid group $\mathcal{B}_{2n}$, which generates the center of the braid group. Therefore, via braid isotopies, we can slide the braid word $B$ across the full twist to get the second plat, defined by $AB\delta$, in Figure~\ref{garside_slide}. Now, we can see that this plat can be resolved via double coset moves to give us the plat defined by $AB$, that we wanted to get to. 

The flip move allows us to turn different strands of the braid both ways simultaneously, which we can not achieve with just the double coset moves.


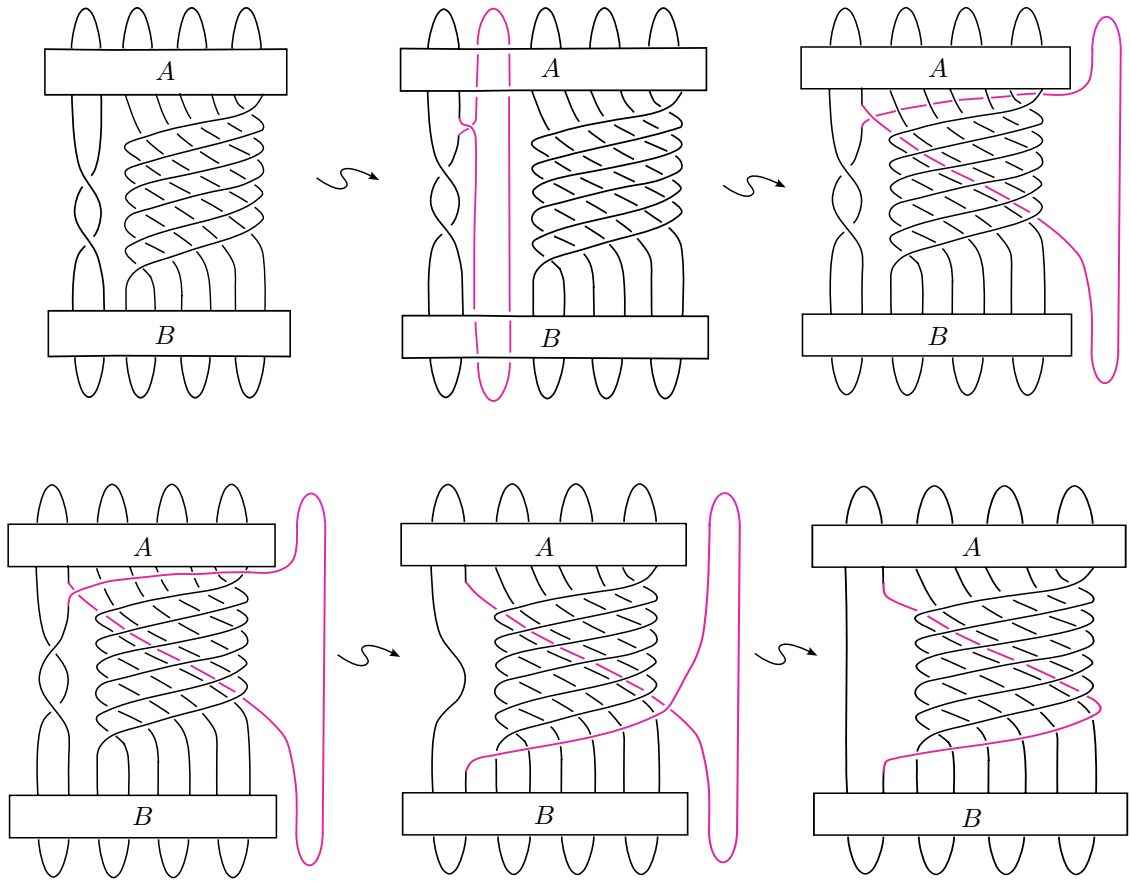
\begin{figure}[ht!]
\centering
\labellist
\small\hair 2pt
\pinlabel $A$ at 120 513
\pinlabel $A$ at 361 515
\pinlabel $A$ at 603 515
\pinlabel $A$ at 106 233
\pinlabel $A$ at 357 233
\pinlabel $A$ at 624 233
\pinlabel $B$ at 120 358
\pinlabel $B$ at 361 356
\pinlabel $B$ at 603 357
\pinlabel $B$ at 106 74
\pinlabel $B$ at 357 75
\pinlabel $B$ at 624 73
\endlabellist
\centering
\includegraphics[width=16cm, height=13cm]{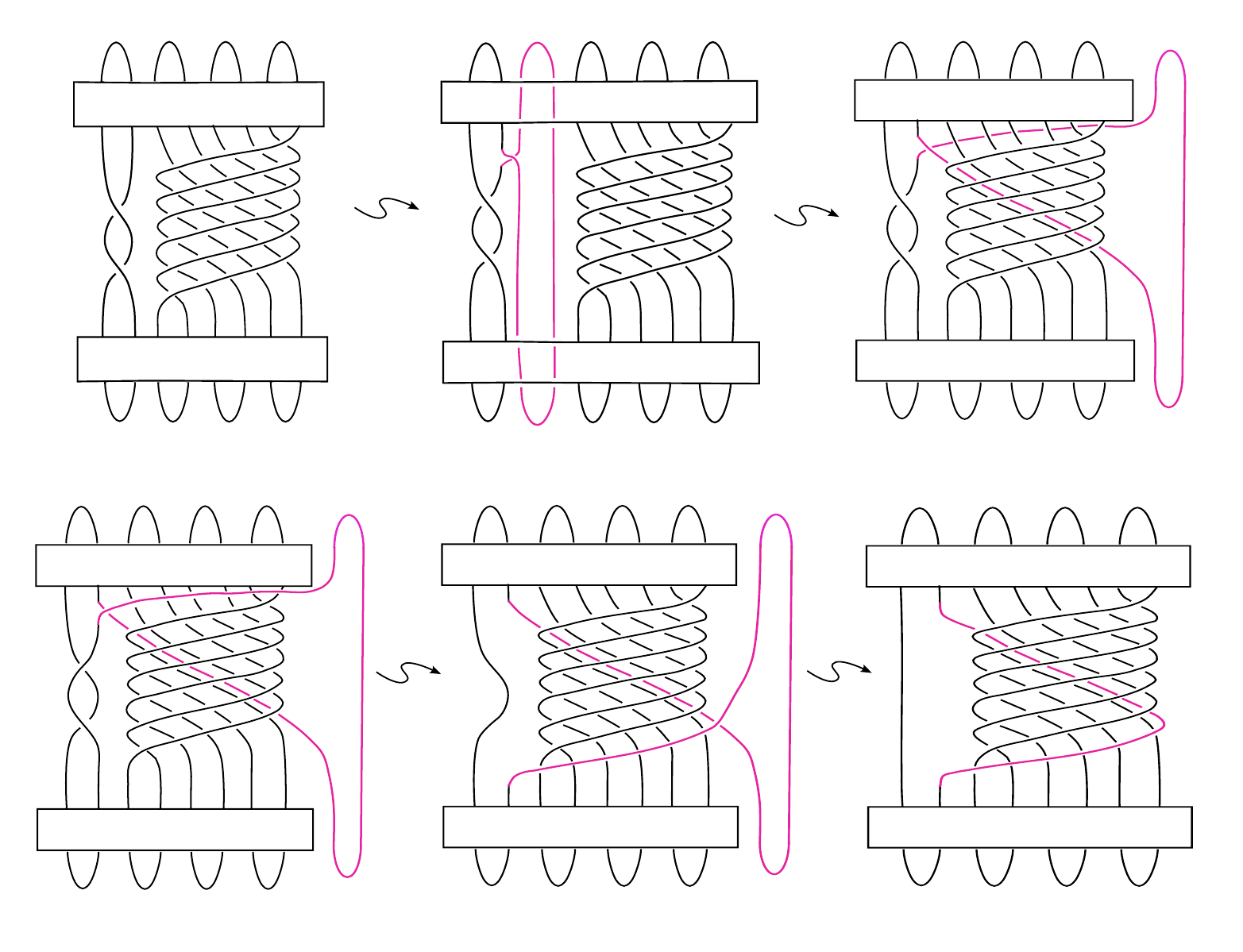}
\caption{Stabilization sequence for the flip move}
\label{stab_seq}
\end{figure}

\begin{figure}[ht!]
\centering
\labellist
\small\hair 2pt
\pinlabel $A$ at 139 683
\pinlabel $A$ at 296 292
\pinlabel $B$ at 140 457
\pinlabel $B$ at 293 95
\pinlabel $AB$ at 436 683
\endlabellist
\centering
\includegraphics[width=14.5cm, height=14.5cm]{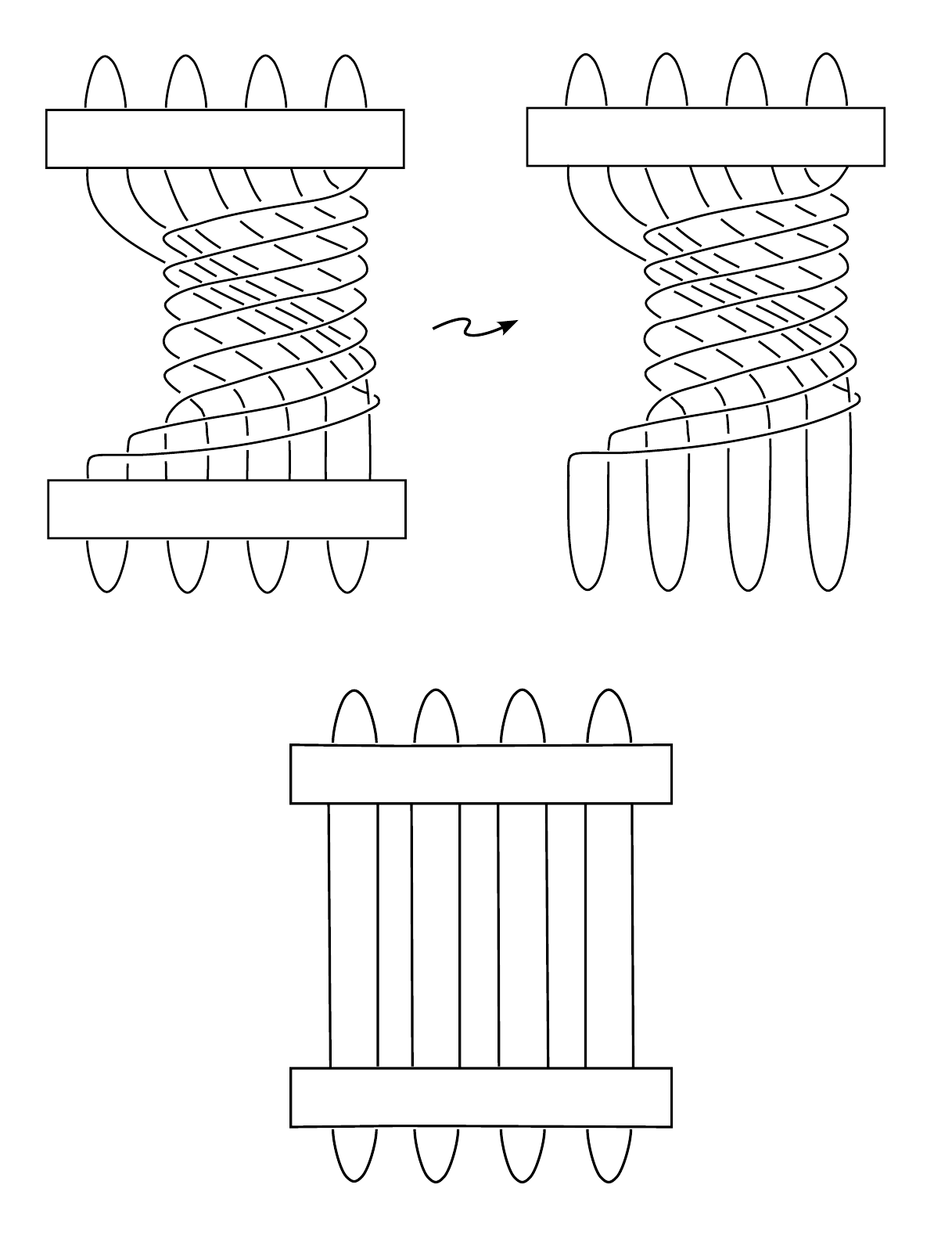}
\caption{Using double coset moves to undo the `turning' of a braid}
\label{garside_slide}
\end{figure}

\subsection{A special configuration of plats} \label{speial_plats}

Figure \ref{std_composite} depicts a special configuration that plats can always be put into. The aim of employing this particular configuration of plats is to have either the rightmost or the leftmost string of the plat devoid of any crossings. We want that so that we can remove singularities from min or max tiles which are attached to a punctured tile. This is explained in detail in Lemma \ref{remove_punctured_tiles}.

\begin{figure}[h!]
\labellist
\small\hair 2pt
\endlabellist
\centering
\includegraphics[width=12cm, height=6cm]{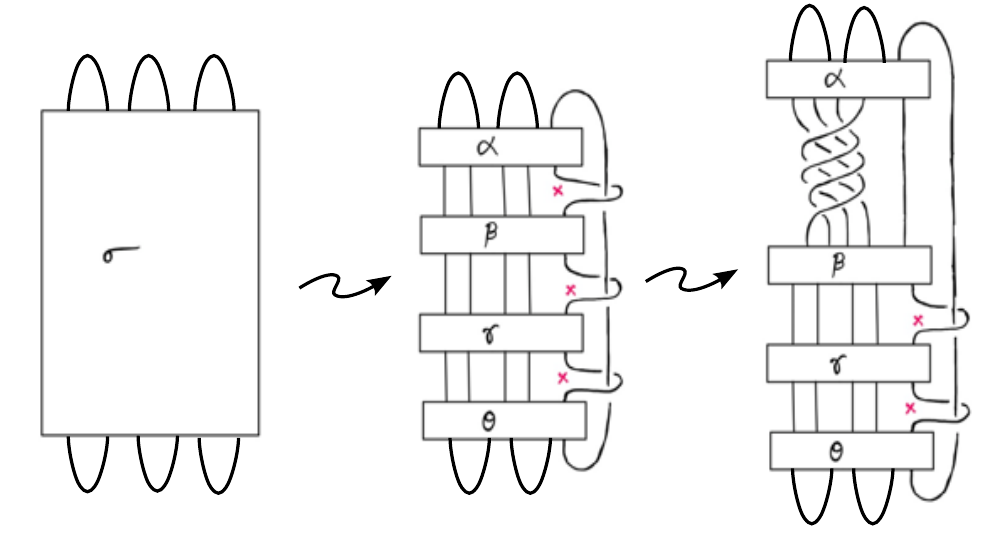}
\caption{A special configuration that every plat can be put into}
\label{std_composite}
\end{figure}

\subsection{Sphere Foliation and Tiling} \label{tilingandfoliation}

As we remarked in \S \ref{level curves}, since the level curves $C_t^i$ are cross sections of $S_i$, taking a union of $C_t^i$'s over all $t$'s gives us the surface $S_i$. Therefore, ${S_i} = \bigsqcup_{t \in [-0.25,1.25]} C_t^i$. This gives a singular foliation of the surface $S_i$, where the leaves of the foliation are the simple closed curves $s_{i,t}^j$'s. The singular leaves are exactly the intersecting (which includes self intersections) curves, of which there are only finitely many. The singular leaves correspond to saddles on the surface. Further, self intersecting closed curves give rise to curves which either self intersect again or degenerate to a single point in the foliation, which corresponds to a local maxima or local minima on $S_i$. As mentioned earlier, these points of maxima are at height $t > 1$ and the minima are at height $t < 0$. 

Using this singular foliation, we will now describe a way to tile $S_i$. We thicken up the singular leaves such that the union of the `neighborhoods' of the singular leaves then determines a tiling of $S_i$. So, a tile $T$ on the sphere $S_i$ is either a thickened up neighborhood of a singular leaf $s$, i.e. a tile $T$ is the surface $s \times I$, where $I$ is the unit interval, or a tile $T$ is a neighborhood of a point of minima or maxima on the sphere. Thus, each tile on the sphere contains exactly one singularity (saddle or local min/max). Abusing notation, for the remainder of the paper, we might refer to a tile and the singularity it contains interchangeably. It will be clear from context what we are referring to. Any tile $T$ has a boundary comprising of circles. A tile with $c$ boundary circles will be denoted as: \bm{$T_c$}. 

Below we give a detailed description of all the possible singularities (and hence, tiles) which show up in the foliation (and tiling) of $S$ and $S_p$. Also, note that each tile is foliated by the level curves $C_t^i$, in a manner so that each tile which is not a neighborhood of a point of minima or maxima, contains exactly one self intersecting singular leaf. With that in mind, we have the following tiles: 

(i) \textbf{Circle intersects circle}: Figure~\ref{gamma} depicts the singularity (and the tile generated from it) \bm{$T_3$}. The top row corresponds to the 3-D embedding on the left and the bottom row corresponds to the 3-D embedding on the right. The tile \bm{$T_3$} is a pair of pants with three boundary circles (hence the subscript `3' in the notation), lying on the sphere. The green shaded portion shown in Figure~\ref{gamma} is NOT a part of the tile.

We classify the saddles corresponding to \bm{$T_3$} into two categories: 

\begin{definition}
a) \textbf{Up saddle}: A saddle such that two curves pass through the saddle to join and become one curve, going in the direction of decreasing `$t$'.\\
In Figure~\ref{gamma}, in the last row, the 3-D embedding on the left depicts an up saddle.

b) \textbf{Down saddle}: A saddle such that one curve passes through the saddle to split into two, going in the direction of decreasing `$t$'.\\
In Figures ~\ref{gamma}, in the last row, the 3-D embedding of the tile, on the right, depicts a down saddle.
\end{definition}

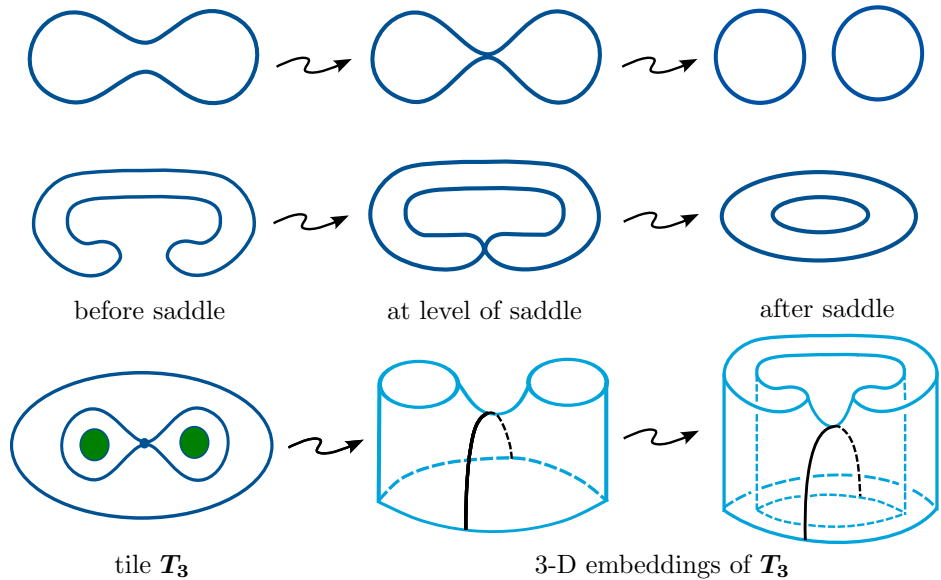
\begin{figure}[h!]
\labellist
\small\hair 2pt
\pinlabel {before saddle}  at 134 208
\pinlabel {at level of saddle} at 368 208
\pinlabel {after saddle} at 607 210
\pinlabel {tile \bm{$T_3$}} at 136 32
\pinlabel {3-D embeddings of \bm{$T_3$}} at 492 32
\endlabellist
\centering
\includegraphics[width=14cm, height=9cm]{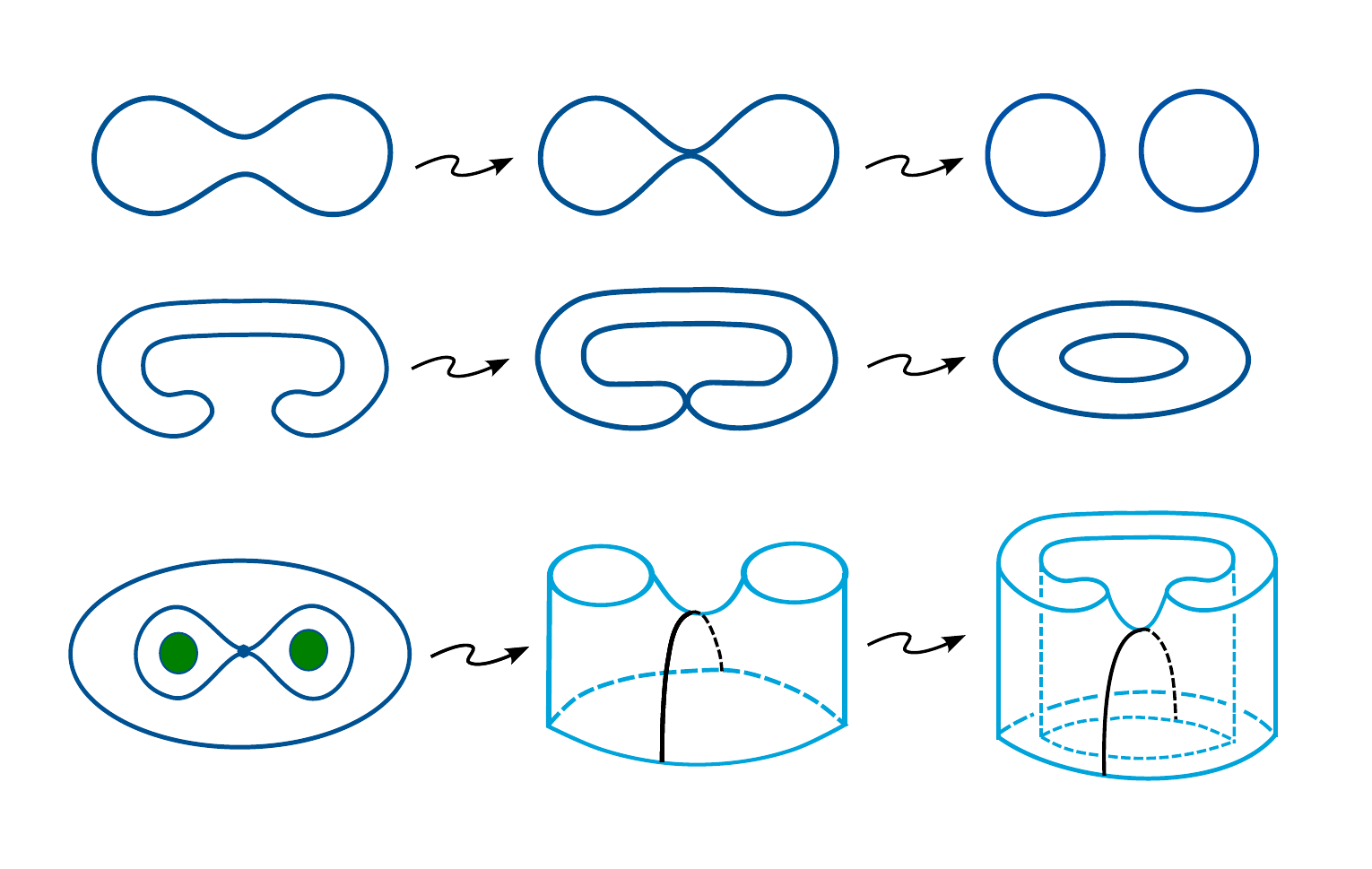}
\caption{Circle intersecting itself - \bm{$T_3$}}
\label{gamma}
\end{figure}

(ii) \textbf{Points of extrema}: The other type of singularity that occurs on the sphere is points of minima and maxima. Name the tile corresponding to these singularities, \bm{$T_1$}, shown in Figure \ref{T_1}. \bm{$T_1$} is a disc, with one boundary circle (hence the subscript `1' in the notation), lying on the sphere. We call a tile of type \bm{$T_1$}, a \textbf{min tile} (respectively, \textbf{max tile}), if the singularity in the tile is a local min (respectively local max).

(iii) \textbf{Punctured disc}: This tile type occurs {\em only} on $S_p$. Call this tile $T_p$. This will essentially be a very small neighborhood of a puncture on $S_p$. By `very small neighborhood', we mean that $T_p$ is a punctured disc containing no other singularities. 

\begin{figure}[h!]
\labellist
\small\hair 2pt
\pinlabel {tile \bm{$T_1$}} at 114 49
\pinlabel {3-D embeddings of \bm{$T_1$}} at 385 44
\endlabellist
\centering
\includegraphics[width=13.5cm, height=6cm]{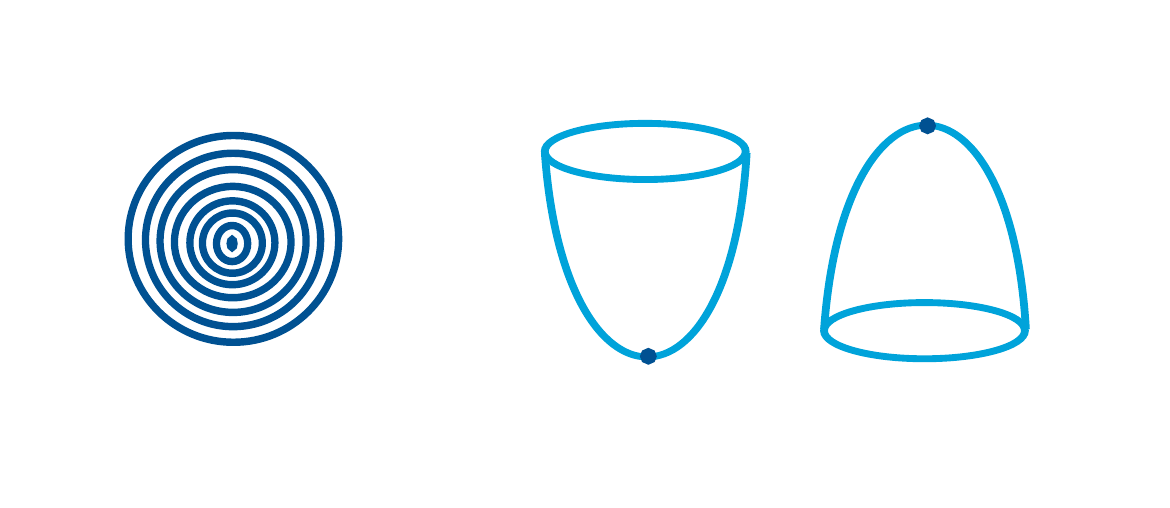}
\caption{Points of extrema - \bm{$T_1$}}
\label{T_1}
\end{figure}

Figures \ref{splitting_sphere} and \ref{twice_punctured_sphere} show the {\em{standard positions}} of the splitting sphere and the twice punctured sphere, respectively. The splitting sphere $S$, in its standard position has two tiles, one max tile and one min tile. On the other hand, the twice punctured sphere $S_p$, in its standard position has a total of six tiles: one each of a max and a min tile, two $T_3$ tiles and two $T_p$ tiles.

\begin{figure}[h!]
\labellist
\small\hair 2pt
\pinlabel {$K_2$} at 303 157
\pinlabel {$K_1$} at 165 160
\pinlabel {max $T_1$ tile} at 292 281
\pinlabel {min $T_1$ tile} at 300 50
\endlabellist
\centering
\includegraphics[width=8cm, height=6.5cm]{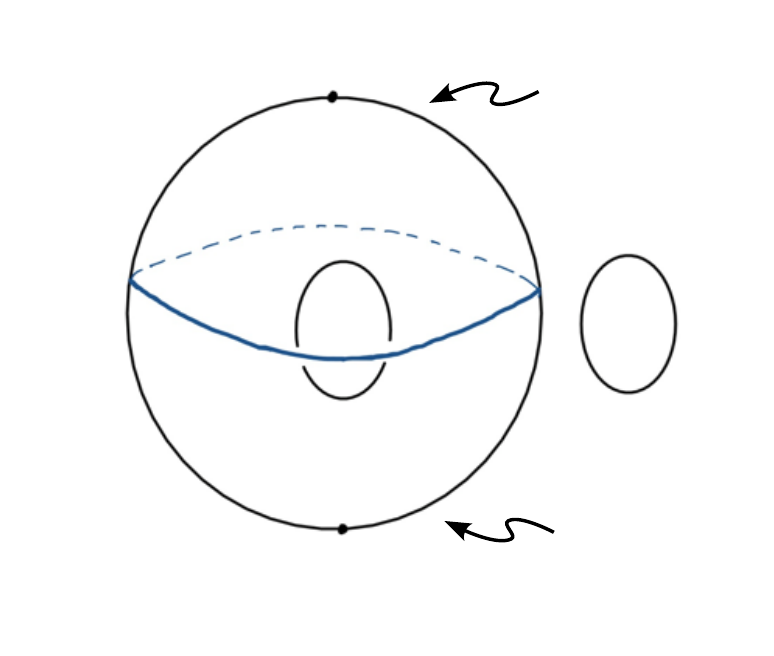}
\caption{Standard configuration of the splitting sphere}
\label{splitting_sphere}
\end{figure}
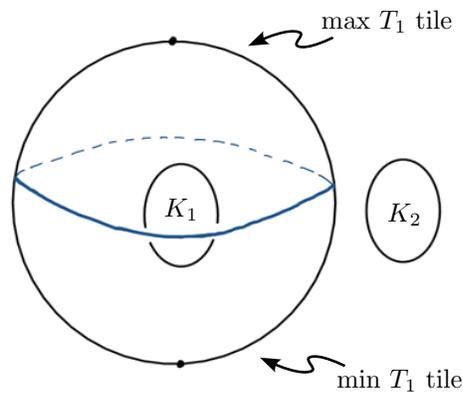

\begin{figure}[h!]
\labellist
\small\hair 2pt
\pinlabel {$T_3$ tile} at 332 222
\pinlabel {$T_3$ tile} at 329 69
\pinlabel {$T_p$ tile} at 241 188
\pinlabel {$T_p$ tile} at 238 104
\pinlabel {max $T_1$ tile} at 284 294
\pinlabel {min $T_1$ tile} at 278 5
\endlabellist
\centering
\includegraphics[width=12.5cm, height=6.5cm]{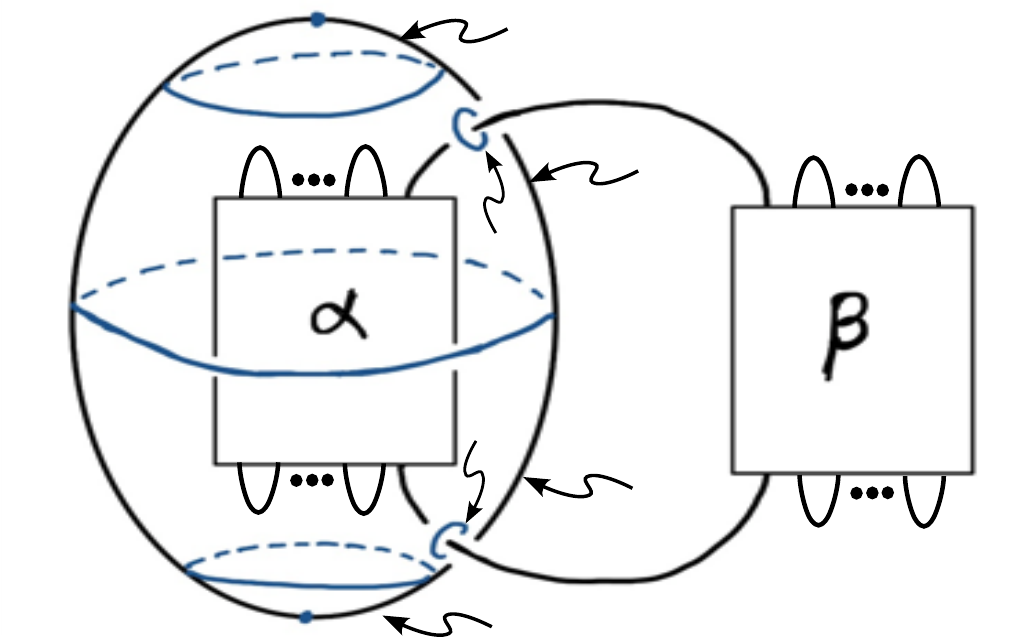}
\caption{Standard configuration of the twice punctured sphere}
\label{twice_punctured_sphere}
\end{figure}

\subsection{Graph and complexity function} \label{graphcomplexity}


Given the tiling of a sphere embedding ${S_i}$, we will now define a directed graph $G_{S_i}$ associated with it in the following way:

The set of vertices of the graph is the set of singularities on $S_i$ or, equivalently, the tiles in the foliation of ${S_i}$. There is an edge going from vertex $a$ to vertex $b$ if the tiles corresponding to the singularities are glued via a circle boundary component in the tiling of the disc and the singularity in the tile $a$ is at a larger $t$-value than the singularity in the tile $b$. Therefore, the valence of a tile $T_c$ is $c$. For the remainder of this paper, abusing notation, when we talk about a vertex of the graph $G_{S_i}$, we might mean either the vertex or the tile represented by it. It should be clear from context, what we mean.

Notice that every edge in the graph corresponds to a simple closed curve, which is a separating curve on the sphere. Now, if we had a cycle in the graph, that would give us a way to build a non-separating curve on the sphere (because the boundary components of each tile $\mbf{\alpha}$ separate the sphere into connected components and each vertex which shares an edge with $\mbf{\alpha}$ has to be in a different connected component of $S_i - \alpha$), which is impossible. Therefore, $G_{S_i}$ {\em is a tree for any embedding $S_i$}. 

We are now in a position to define the complexity function. Let $|T_c|$ be the number of type $T_c$ tiles in the tiling of $S_i$. Then, define the complexity function on $S_i$ as follows: 

$$ c(\mathbf{S_i}) = |T_3|$$ 
Lemmas \ref{bla} and \ref{bla_puncture} in \S \ref{mainresult} tells us why it makes sense to define the complexity function this way.

\section{\centering Main Result}  \label{mainresult}

We prove the following Lemmas which we will later use to prove Theorems \ref{thm1} and \ref{thm2}: 

\begin{lemma}

Let $X_i, i = 1, 2, ..., n$ be $n$ simplicial complexes of degree 2 which are glued together along boundary circles, to make a surface $X$. Let there be $|X_i|$ many $X_i$'s such that each such piece is glued in the same manner. Let $\chi(X_i)$ be the Euler characteristic of $X_i$. Then, we have:

$$ \chi(X) = \sum_{i=1}^{n} |X_i| \chi(X_i) $$

\label{surfaceec}

\end{lemma}

\begin{proof}

For any simplicial complex $C$, let $n_i(C)$ be the number of $i$-cells in $C$. \\
Consider the case where $X$ is obtained from $X_1$, $X_2$ by gluing $X_1$ with $X_2$ via circles. Let $\tau$ be the gluing simple closed curve.\\
Now, let $X$ have a cell decomposition in which the circle $\tau$ is a subcomplex, and so $\tau$ consists of $k$ 0-dimensional cells and $k$ 1-dimensional cells for some $k \geq 1$. \\
Cutting $X$ along $\tau$ to get $X_1 \bigcup X_2$ and counting cells, we get:

\begin{equation}
n_0(X) = n_0(X_1) + n_0(X_2) - k
\end{equation}  

\begin{equation}
n_1(X) = n_1(X_1) + n_1(X_2) - k
\end{equation}

\begin{equation}
n_2(X) = n_2(X_1) + n_2(X_2)
\end{equation}

(1) - (2) + (3) gives us: 
$$ \chi(X) = \chi(X_1) + \chi(X_2) $$

Induction completes the proof.
 
\end{proof}

\begin{lemma}
Given any geometric realisation and tiling of the splitting sphere $S$, the following combinatorial relation holds:\\
 $|T_1| - |T_3| = 2$, \\
where $|T_i|$ is the number of tiles of type $T_i$ in the foliation of $S$.

\label{bla}
\end{lemma}

\begin{proof}

From Lemma~\ref{surfaceec}, we have the following equation: 

\begin{align*}
 \chi(X) &= \sum_{i=1}^{n} |X_i| \chi(X_i) \\
\end{align*}

Reconstructing $S$, gluing together tiles of type $T_1$ and $T_3$, and using the equation above, we get the following relationship:

\begin{align*}
2 &= |T_3|\chi(T_3) + |T_1|\chi(T_1)\\
  &= |T_3|(-1) + |T_1|(1) \\
  &= |T_1| - |T_3|  
\end{align*}
\end{proof}

We have a similar result for $S_p$, with slight modifications. We consider the compact subsurface of $S_p$ which consists of only tiles of type $T_1$ and $T_3$. It is enough to consider this because the number of tiles of type $T_p$ on {\em any} geometric realisation of $S_p$ is fixed, and is equal to two. This sub-surface is an annulus, say $\mbf{A_p}$, because this is essentially a compact sub-surface of a twice punctured sphere, with two boundary components. Therefore, for $S_p$, we get the following result: 

\begin{lemma}
Given any geometric realisation and tiling of the twice punctured sphere $S_p$, the following combinatorial relation holds:\\
 $|T_1| = |T_3|$, \\
where $|T_i|$ is the number of tiles of type $T_i$ in the foliation of $S_p$.

\label{bla_puncture}
\end{lemma}

\begin{proof}

From Lemma~\ref{surfaceec}, we have the following equation: 

\begin{align*}
 \chi(X) &= \sum_{i=1}^{n} |X_i| \chi(X_i) \\
\end{align*}

Reconstructing $A_p$ as a union of $T_1$'s and $T_3$'s, we get the following relation:

\begin{align*}
\chi(A_p) = 0 &= |T_3|\chi(T_3) + |T_1|\chi(T_1)\\
  &= |T_3|(-1) + |T_1|(1) \\
  &= |T_1| - |T_3|  
\end{align*}

Now, the observation that $S_p = A_p \bigcup T_p \bigcup T_p$ completes the proof.

\end{proof}

We are now ready to prove Theorem \ref{thm1}:

\begin{proof}
Lemma \ref{bla} tells us that removing one singularity of type $T_3$ also gets rid of one singularity of type $T_1$, and vice versa. It also tells us that there will always be more than 2 singularities of type $T_1$ in any foliation of any $S_i$. Notice that for a plat presentation of a split link with a split diagram, the sphere has only two tiles: a min $T_1$ and a max $T_1$, this is shown in Figure \ref{splitting_sphere}. As mentioned before, this is called the {\em standard configuration} of $S$. With this setup, to simplify a plat presentation $L_i$ of a split link to an `obviously' split plat presentation, all we need to do is reduce the corresponding sphere $S_i$ to its standard configuration. Towards achieving that goal, our strategy going forward is to remove singularities on $S_i$ by doing plat moves on the split link. Note that we employ the same strategy to get the sequence for Theorem \ref{thm2} i.e. our aim is to reduce the geometric realisation of $S_p$ corresponding to a `complicated' plat presentation of a composite knot to an `obviously' composite one. We start out by removing any possible crossings in the link diagram via braid isotopies or double coset moves. Then, we proceed to removing singularities from the sphere foliation using the pocket move in a fashion so that $c(S_{i+1}) < c(S_i)$. We stop this process when we reach some $L_k$ so that $c(S_k) = 0$. We will use the directed graph defined in the previous section to help keep track of the process of removing singularities from the sphere foliation. Toward that end, the next two Lemmas tell us when can a vertex (i.e. tile or singularity) be removed from the graph using plat moves: 

\begin{lemma}

For a given $L_i$ in the sequence of Theorems \ref{thm1} or \ref{thm2}, in the directed graph $G_{S_i}$, if there is a vertex $V$ which satisfies one of the following conditions: 

a) $V$ represents a min tile of type \bm{$T_1$} connected to a vertex $N$ representing a tile containing a down saddle, such that any level curve corresponding to $V$ does not contain any other level curve 

b) $V$ represents a max tile of type \bm{$T_1$} connected to a vertex $U$ representing a tile containing an up saddle such that any level curve corresponding to $V$ does not contain any other level curve then, we can remove the vertex $V$ using the pocket move, so that we get to $L_{i+1}$ in the sequence given in Theorem 1, where $c(S_{i+1}) < c(S_i)$.

\label{lemma5}

\end{lemma}

\begin{proof}

\begin{figure}[h!]
\labellist
\small\hair 2pt
\pinlabel {$\alpha$}  at 193 90
\pinlabel {$V$} at 121 46
\pinlabel {$V_1$} at 230 46
\pinlabel {$\mbf{B}$} at 110 236
\pinlabel {$S_{V_1}$} at 177 170
\pinlabel {cross section of $N(\alpha)$} at 400 129
\pinlabel {$N(\alpha)$} at 265 180
\endlabellist
\centering
\includegraphics[width=11cm, height=8cm]{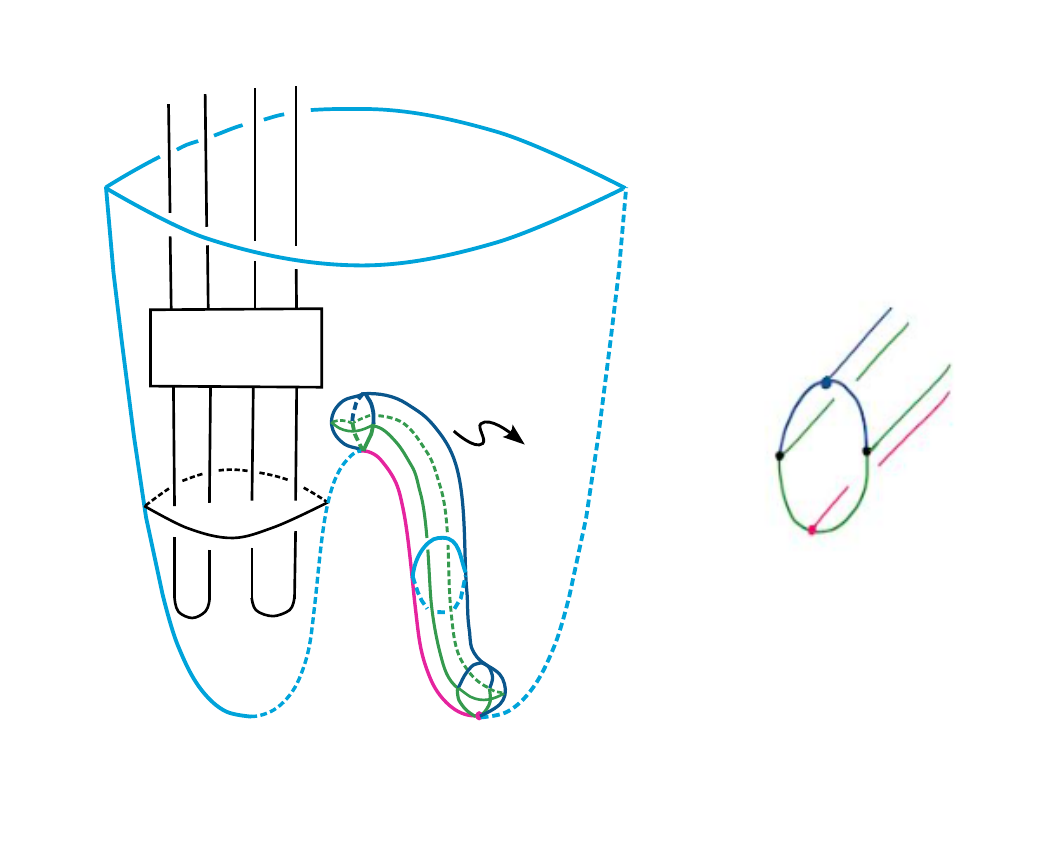}
\caption{Path isotopy to remove saddles}
\label{path_isotopy}
\end{figure}

\begin{figure}[h!]
\labellist
\small\hair 2pt
\endlabellist
\centering
\includegraphics[width=12cm, height=5cm]{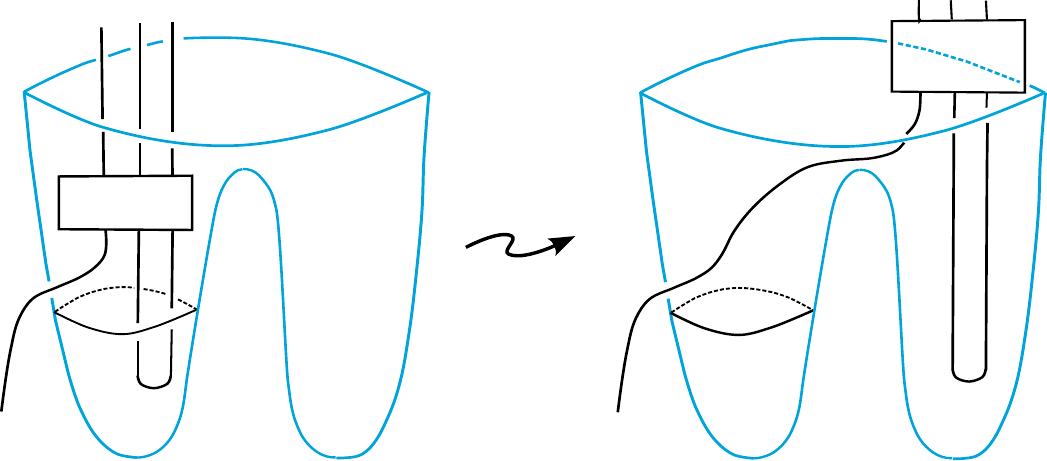}
\caption{Removing saddle attached to the tile containing the puncture}
\label{saddle_punctured_sphere}
\end{figure}

\begin{figure}[h!]
\labellist
\small\hair 2pt
\endlabellist
\centering
\includegraphics[width=12cm, height=5cm]{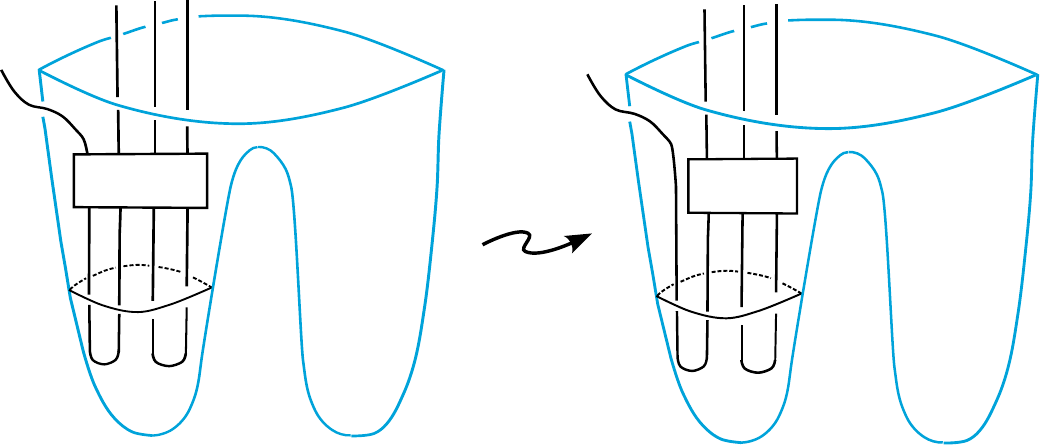}
\caption{Removing saddle attached to the tile containing the puncture}
\label{flip_punctured_sphere}
\end{figure}

Consider a vertex $V$ satisfying condition a). Then, we look for a closest min tile to $V$, in the graph $G_{S_i}$, say $V_1$. Note that such a tile always exists because $V$ is connected to a tile of type $T_3$, say $N$, containing a down saddle, where $N$ is either connected to a min tile of type $T_1$, or $N$ is connected to another tile of type $T_3$, which might contain either an up saddle or a down saddle. In either case, this process continues until we hit another min tile, because the graph $G_{S_i}$ is a finite tree. Consider the saddle $N_{V}$ contained in tile $N$, that the singularity in tile $V$ is induced by. Now, define a path $\alpha$ (the pink curve in Figure \ref{path_isotopy}), from $N_{V}$ to $V_1$, on the sphere, such that $\alpha$ is always transverse to the foliation, other than at points $\alpha(0)$, $\alpha(1)$. Consider a thickened up regular neighborhood of $\alpha$, say $N(\alpha)$ which stays completely on one side of the sphere, as shown in Figure \ref{path_isotopy}. Notice that $N(\alpha)$ is a 3-ball, positioned so that the boundary sphere is split up into the following four components: two copies of $\alpha \times I$ (the green part in the figure which lives on the surface, and the blue part which is completely on one side of the surface) glued along their respective $\alpha \times \{0\}$ and $\alpha \times \{1\}$ components. This gives us an annulus. Capping off this annulus with two discs gives us the two sphere which is the boundary of $N(\alpha)$.

For any singularity $X$ on $S_i$, let $t_{X}$ denote the level at which the singularity occurs. Now, we inspect $C^{i}_{t}$, for some $t \in (t_N, t_N+\epsilon)$. The knot strands inside the curves corresponding to $V$, we remove via the following isotopy: we move the bridges corresponding to these strands along the path $\alpha$, always staying inside the $3$-ball $N(\alpha)$, starting from the point $N_V$ on the boundary of $N(\alpha)$ and ending at the point $V_1$, also on the boundary of $N(\alpha)$. Notice that this corresponds to doing a { \em pocket move}. Once we have emptied out the simple closed curve corresponding to $V$, we can surger off the subdisc which is the tile represented by $V$, thus simultaneously cancelling the saddle $N$, while removing the vertex $V$. \\
We can do a similar isotopy for a vertex $V$ satisfying condition b), with the only difference being that we now need to look for the next closest max tile to $V$.

In both the cases above, we get a new plat $L_{i+1}$ such that $c(S_{i+1}) < c(S_i)$.

\end{proof}

\begin{lemma}
A vertex satisfying condition a) or b) in Lemma \ref{lemma5} always exists, if there are any tiles of type \bm{$T_3$} in the foliation of the sphere embedding $S_i$.

\label{lemma6}

\end{lemma}

\begin{proof}

\begin{figure}[h!]
\labellist
\hair 6pt
\pinlabel {connecting scc}  at 158 156
\pinlabel {up saddle} at 453 175
\pinlabel {down saddle} at 55 225
\pinlabel {$s_1$} at 205 206
\pinlabel {$s_2$} at 311 170
\pinlabel {$b$} at 155 178
\pinlabel {$c$} at 292 187
\pinlabel {$d$} at 385 200
\endlabellist
\centering
\includegraphics[width=11cm, height=8cm]{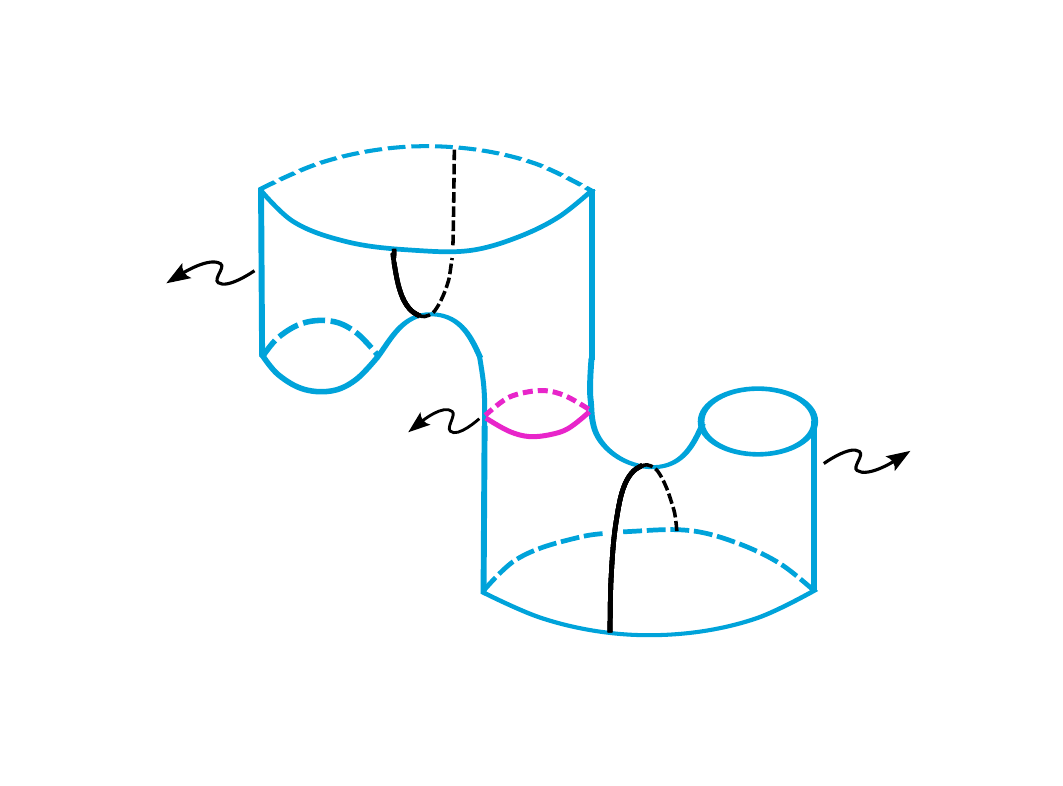}
\caption{A down saddle connected to an up saddle via a simple closed curve}
\label{updownsaddle}
\end{figure}
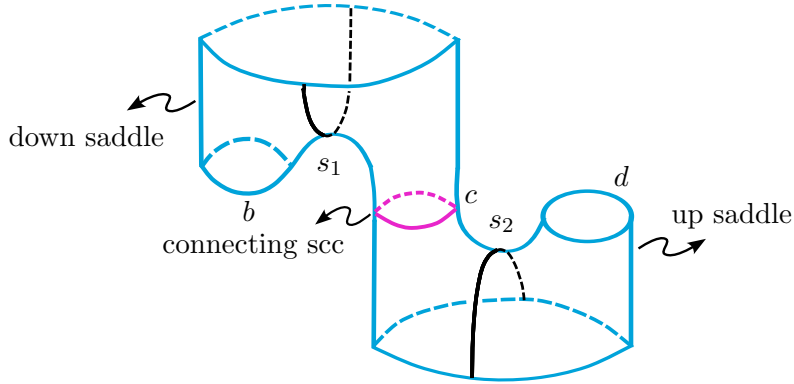

We argue this by contradiction. If there were no such vertices as described in the previous lemma, then no down saddle is connected to a min tile and no up saddle is connected to a max tile.

Then, near any down saddle $s_1$, the splitting sphere locally looks like as depicted in Figure~\ref{updownsaddle}, with the variation that $s_1$ could also have been connected to another down saddle, as opposed to the up saddle shown in the figure. Now, since the saddle $s_1$ is not connected to a min tile, neither of the two curves $b$, $c$ are capped off by discs. Therefore, $s_1$ is connected to two more saddles, say $s_2$ and $s_3$. Note that $s_2$ and $s_3$ cannot be connected to each other because that would mean the surface has genus, which is a contradiction. Then, if both $s_2$ and $s_3$ are down saddles, by the same argument, as above, we get two more down saddles each for $s_2$ and $s_3$, because they are not connected to any min tiles either. If this process continues indefinitely, we have infinitely many tiles, which can not happen.

Therefore, without loss of generality, let $s_2$ be an up saddle. Now the local picture is exactly as shown in Figure~\ref{updownsaddle}. Further, since no up saddle is connected to a max tile, $d$ cannot be capped off by a disc, therefore $d$ is further connected to another saddle and, as before, this process continues indefinitely unless there is an up saddle connected to a max tile or down saddle connected to a min tile. \\
If we locate a min tile $M$ connected to a down saddle $s_1$, such that the level curve corresponding to $M$ contains a level curve coming from another tile connected to $s_1$, then we use the pocket move to remove the inner level curve. Doing this, we now have a vertex satisfying condition a) of the previous lemma. The argument is exactly the same for a max tile.\\
\end{proof}

In light of Lemmas \ref{lemma5} and \ref{lemma6}, we now have a blueprint which tells us how to create the sequence in Theorem \ref{thm1}, monotonically decreasing the complexity function on the $L_i$'s. We terminate the sequence when we reach a $L_l$ which does not have any singularities on the splitting sphere, other than a top max tile and a bottom min tile. Then, any remaining crossings come from twisting the bridges or doing Reidemeister II moves or moving bridges completely over or under each other, which can be undone using the double coset moves or braid isotopies, leading us to an `obviously' split presentation. This completes the proof of Theorem \ref{thm1}.

\end{proof}

The proof for Theorem \ref{thm2} is very similar to that of \ref{thm1}. We remove tiles of type $T_3$ from $S_p$ in the same fashion that we did for $S$, except for when removing the saddle from a type $T_3$ tile which is also attached to a $T_p$ tile. To deal with that case, we utilise the following result:

\begin{lemma} \label{remove_punctured_tiles}
    If the punctured sphere $S_p$ has a tile of type $T_3$ with a saddle such that it is connected to both a tile of type $T_1$ and $T_p$, then we can remove the saddle and a $T_1$ tile corresponding to the $T_3$ tile using the moves detailed in Theorem \ref{thm2}.  
\end{lemma}

\begin{proof}
This type of situation can happen in the following different ways: 

Case I: The first case is the configuration shown in Figure \ref{saddle_punctured_sphere}. Then, as shown in the Figure, we can do a pocket move which empties out the specified $T_1$ tile. Then, we can do an isotopy on the surface to remove the saddle corresponding to the $T_3$ tile and we are done.

Case II: The second case is the configuration shown in Figure \ref{flip_punctured_sphere}. To remove the singularity in this case, we utilise the set up detailed in \S \ref{speial_plats}. Using the flip move, we first do flip moves at the appropriate spots (in the example shown in Figure \ref{std_composite}, we do the flip move at the pink crosses). This enables to free up the outermost strand (this will be the left most or right most depending on where the puncture is located). Once we have achieved that configuration, we again do a pocket move to empty out the specified $T_1$ tile and then do an isotopy to remove the saddle.

Case III: This is when a saddle containing $T_3$ tile is attached to two $T_p$ tiles. Then we consider the third boundary circle of the $T_3$ tile. It is either connected to another $T_3$ tile or a $T_1$ tile. In either event, we can find some $T_1$ tile and a corresponding saddle to be removed, via the process described for the splitting sphere.

This completes the proof.

\end{proof}

\printbibliography 

\end{document}